\documentclass{amsart}
\usepackage[T1]{fontenc} 
\usepackage[utf8]{inputenc}
\usepackage[backend = biber, style = alphabetic, natbib = true, backref = true, doi = false]{biblatex}
\usepackage{microtype}
\usepackage{lmodern}
\usepackage{amsmath, amsfonts, amssymb}
\usepackage{tikz}
\usepackage{tikz-cd}
\usetikzlibrary{matrix,arrows,positioning}
\tikzset{
 /tikz/commutative diagrams/arrow style = tikz }
\usepackage{hyperref}
\usepackage{amsthm}
\usepackage{cleveref}

\title{On Global Deformations of Quartic Double Solids}
\author{Tobias Dorsch}
\thanks{This paper is an abridged, restructured version of my PhD thesis. I would like to thank my advisor Thomas Peternell for many helpful discussions and continuous support. I would also like to thank Vlad Lazić for many comments, suggestions and insistences concerning the presentation. This work was funded by the DFG-Forschergruppe 790 ``Classification of Algebraic Surfaces and Compact Complex Manifolds''.}

\numberwithin{equation}{section}
\newtheorem{thm}{Theorem}[section]
\newtheorem{prop}[thm]{Proposition}
\newtheorem{lem}[thm]{Lemma}
\newtheorem{assum}[thm]{Assumption}
\newtheorem{stmt}[thm]{Statement}

\newtheorem{que}[thm]{Question}
\newcounter{clm}[thm]
\newtheorem{claim}[clm]{Claim}

\theoremstyle{definition}
\newtheorem{defn}[thm]{Definition}
\newtheorem{rem}[thm]{Remark}

\crefname{assum}{Assumption}{Assumptions}
\crefname{stmt}{Statement}{Statements}
\crefname{thm}{Theorem}{Theorems}
\crefname{lem}{Lemma}{Lemmas}
\crefname{lemdef}{Lemma}{Lemmas}
\crefname{prop}{Proposition}{Propositions}
\crefname{claim}{Claim}{Claims}
\crefname{notation}{Notation}{Notations}
\crefname{que}{Question}{Questions}
\crefname{section}{Section}{Sections}

\newcommand{\C}{\mathbb{C}}

\newcommand{\Q}{\mathbb{Q}}
\newcommand{\Z}{\mathbb{Z}}
\newcommand{\N}{\mathbb{N}}

\newcommand{\Proj}[1]{\mathbb{P}(#1)}
\newcommand{\Pn}[1]{\mathbb{P}^{#1}}
\newcommand{\Qn}[1]{Q^{#1}}
\newcommand{\hir}[1]{\mathbb{F}_{#1}}
\newcommand{\idealf}[1]{\mathcal{#1}}
\newcommand{\ideal}[1]{\idealf{I}_{#1}}

\newcommand{\Pic}[1]{\operatorname{Pic}(#1)}

\newcommand{\sO}[2]{\mathcal{O}_{#1}(#2)}
\newcommand{\sg}[1]{\mathcal{O}_{#1}}
\newcommand{\can}[1]{K_{#1}}
\newcommand{\dualizing}[1]{\omega_{#1}}

\newcommand{\conormal}[2]{N^{\vee}_{#1/#2}}

\newcommand{\eChar}[2]{\chi(#1,#2)}

\newcommand{\bigX}{\mathfrak{X}}
\newcommand{\bigL}{\mathcal{L}}
\newcommand{\Com}[3]{H^{#1}(#2,#3)}
\newcommand{\com}[3]{h^{#1}(#2,#3)}
\newcommand{\betti}[2]{b_{#1}(#2)}
\newcommand{\cCoh}[3]{H_c^{#1}(#2,#3)}

\newcommand{\hdi}[3]{R^{#1}{#2}_*#3}

\newcommand{\bs}[1]{\operatorname{Bs}(#1)}

\newcommand{\exc}[1]{\operatorname{Exc}(#1)}
\newcommand{\supp}[1]{\operatorname{supp}(#1)}

\newcommand{\CInf}{C_{\infty}}

\newcommand{\eInf}{{\mathsf{e}}_{\infty}}

\newcommand{\e}{\mathsf e}
\newcommand{\f}{\mathsf f}

\newcommand{\sing}[1]{\operatorname{Sing}(#1)}

\newcommand{\mult}[2]{\operatorname{mult}_{#1}{#2}}

\newcommand{\fun}[3]{#1\colon #2 \to #3}

\newcommand{\deq}{: \hspace{-0.02cm}=}

\bibliography{global_def}

\begin{document}
\begin{abstract}
It is shown that a smooth global deformation of quartic double solids, i.e.\ double covers of $\Pn 3$ branched along smooth quartics, is again a quartic double solid without assuming the projectivity of the global deformation. The analogous result for smooth intersections of two quadrics in $\Pn 5$ is also shown, which is, however, much easier. \par
In a weak form this extends results of J.\ Kollár and I.\ Nakamura on Moishezon manifolds that are homeomorphic to certain Fano threefolds and it gives some further evidence for the question whether global deformations of Fano manifolds of Picard rank $1$ are Fano themselves. 
\end{abstract}
\maketitle
\section{Introduction}
The aim of this paper is to show a rigidity result for certain Fano $3$-folds. Mainly we want to show that a smooth global deformation of quartic double solids is again a Fano manifold. More precisely the main result is the following theorem.
\begin{thm}\label{mainThm}
 Let $\fun{\pi}{\bigX}{\Delta}$ be a smooth family of compact complex manifolds over the unit disc $\Delta \subset \C$. Suppose that there is a sequence $(s_n)_{n \in \N}$, $s_n \in \Delta$, with $s_n \to 0$ such that the fibres $\bigX_{s_n} = \pi^{-1}(s_n)$ are Fano manifolds for all $n \in \N$ and suppose that for one $n \in \N$ (hence for all) the Fano manifold $\bigX_{s_n}$ is a quartic double solid, i.e.\ a double cover of $\Pn 3$ branched along a smooth quartic. \par
 Then $\bigX_0 = \pi^{-1}(0)$ is a Fano manifold (and again a quartic double solid).
\end{thm}
Given the assumptions of \cref{mainThm} we call $\bigX_0$ a global deformation of quartic double solids and similarly for other classes of (Fano) manifolds. \par 
The following proposition is much easier to prove than \cref{mainThm}.
\begin{prop}\label{easyProp}
A global deformation of Fano threefolds that are intersections of two quadrics in $\Pn 5$ is a Fano manifold (and again an intersection of two quadrics in $\Pn 5$).
\end{prop}
Of course, both results are essentially obvious if we additionally assume that all fibres of $\pi$ are projective. They are motivated by the following more general question. 
\begin{que}\label{q1} 
Is a global deformation of Fano manifolds with Picard number $1$ again a Fano manifold?
\end{que}
If the Picard number is greater than $1$, this is in general false, as can already be seen from the example of $\Pn 1 \times \Pn 1$ degenerating to the second Hirzebruch surface $\hir 2$. \par 
Looking again at Fano manifolds with Picard number $1$, there are some indications that there is a positive answer to \cref{q1}: Y.-T. Siu showed that, for any $n$, a global deformation of projective spaces $\Pn n$ is again $\Pn n$ (\cite{siu89},\cite{siu89b},\cite{siu89c}) and J.-M. Hwang showed the analogous result for smooth hyperquadrics $\Qn n$ of dimension $n \geq 3$ (\cite{hwang95}). \par In the view of the result of S. Kobayashi and T. Ochiai that a Fano manifold $X$ of index at least $n = \dim X$ is either $\Qn n$ or $\Pn n$, the results of Siu and Hwang just say that \cref{q1} has a positive answer, if one additionally assumes that the index of the degenerating manifolds is not less than their dimension. \par
In dimension $3$ the following slightly more general question has been studied. 
\begin{que}\label{q2}
If a (compact) Moishezon threefold is homeomorphic to a Fano manifold with Picard number $1$, is it then itself Fano (and of the same type)?
\end{que}
Note that this is, indeed, more general, since a global deformation of Fano manifolds of a certain type is homeomorphic to a Fano manifold of that same type by Ehresmann's theorem and its anticanonical bundle is big, hence the manifold is Moishezon, by the Semi Continuity Theorem. \par 
\cref{q2} has been positively answered by J. Kollár in the case when the manifold is homeomorphic to the intersection of $Gr(2,5) \subset \Pn 9$, in its Plücker embedding, with a linear subspace of codimension $3$ -- and by I. Nakamura in the case when the manifold is homeomorphic to a cubic threefold. \par
\cref{mainThm} gives a positive answer to \cref{q1} under the assumption that $X$ is a global deformation of quartic double solids and \cref{easyProp} under the assumption that it is a global deformation of Fano threefolds that are intersections of two quadrics in $\Pn 5$. \par
Let us briefly survey the proofs. One easily sees that there exists $L \in \Pic X$ such that 
\[ 
\Pic X = \Z L, \quad -\can X = 2L, \quad L \text{ big}
\] 
where we set $X = \bigX_0$. The aim is to show that $L$ is ample and the idea is to show that $|L|$ is base point free and that the induced morphism $\Phi$ is finite. This then certainly implies the claim, since the pullback of an ample line bundle by a finite morphism is ample. \par 
Assuming that $|L|$ is base point free, an easy argument shows the finiteness of $\Phi$ (cf.\ \cref{bpf-ample}): first we show that $\Phi$ is generically finite, i.e.\ that it has $3$-dimensional image. This is essentially because each element $D \in |L|$ is irreducible and because for all distinct divisors $D_1,D_2 \in |L|$, the curve $D_1 \cap D_2$ is connected. Since $\Pic X \simeq \Z$, the map $\Phi$ cannot contract a divisor. Suppose there exists an integral curve $C \subset X$ contracted by $\Phi$. One shows that $C \simeq \Pn 1$ and by definition of $\Phi$ we have $\can X \cdot C = 0$. Now we use the fact that $X$ is a global deformation, i.e.\ that it lives in a smooth family of complex manifolds over the unit disc $\Delta$. The deformation theory of smooth rational curves implies that $C \subset \bigX$ deforms in a family of dimension at least $1$. Since $-\can {\bigX_{s_n}}$ is ample, this shows that there are infinitely many numerically trivial curves in $X = \bigX_0$ and it is known that this is a contradiction (cf.\ \cref{notInfinite}). \par 
In order to show that $|L|$ is  base point free, we use a result by Kollár. He gives a structure theorem for Moishezon manifolds $X$ with\footnote{In fact his assumptions are a priori slightly weaker.} 
\[
 \Pic X =  \Z L, \qquad -\can X = 2 L, \qquad L \text{ big}.
\]
We have already seen that if the map $\Phi$ induced by $|L|$ is a morphism, then it has $3$-dimensional image. By a more careful analysis Kollár shows that the same is true in general. He additionally shows that if $\com{0}{X}{L} \geq 4$ and if $\Phi$ is not a morphism, then either 
\begin{enumerate} 
\item $\Phi$ is a bimeromorphism to $\Pn 3$, or 
\item $\Phi$ is a bimeromorphism to a smooth quadric $\Qn 3 \subset \Pn 4$.
\end{enumerate}
In particular, $|L|$ is always base point free if $\com{0}{X}{L} > 5$. In our global deformation setup we have $\com{0}{X}{L} \geq \com{0}{\bigX_{s_n}}{\bigL_{s_n}}$ by upper semi continuity. \par 
If $X$ is a global deformation of intersections of two quadrics in $\Pn 5$, this already shows that $|L|$ is base point free, hence $L$ is ample by what we have seen before, which proves \cref{easyProp}.\par
If $X$ is a global deformation of quartic double solids, we still have $\com 0 X L \geq 4$. In this case we finish the proof of \cref{mainThm} by showing that the cases (1) and (2) from above actually cannot occur. Concerning (2), we closely follow Nakamura's strategy from his proof that a Moishezon manifold homeomorphic to a cubic threefold is a cubic threefold \cite{Nak}. Concerning (1), many of Nakamura's ideas are still helpful, but there are also new phenomena, which need new ideas. \par 
Slightly more precisely, in \cref{quadric} and \cref{projective} we show that if $X$ is a global deformation of quartic double solids and if $\Phi$ is a bimeromorphism to $\Pn 3$ or $\Qn 3$, then  
\[
 \betti 3 X \leq 12.
\]
This is a contradiction to the fact that $\betti 3 X = 20$, which holds by Ehresmann's theorem and because it is true for quartic double solids.  
\section{Notation and preliminary results}
When $X$ is a complex space and $A,B \subset X$ are closed complex subspaces with ideal sheaves $\ideal A, \ideal B \subset \sg X$, we let 
\[
 A \cdot B \qquad \text{resp.} \qquad A + B
\] 
denote the closed complex subspaces corresponding to the ideal sheaves 
\[
 \ideal A + \ideal B \qquad \text{resp.} \qquad \ideal A \cap \ideal B.
\]
When $L$ is a line bundle on a compact complex manifold $X$ and $S$ is some statement, then a phrase like 
\begin{center}
 ``For general $D_1,D_2 \in |L|$, $S(D_1,D_2)$ holds.''
\end{center} 
is supposed to mean 
\begin{center}
``There exists a Zariski-open, Zariski-dense subset $U \subset |L|$ and for each $D_1 \in U$ there exists a Zariski-open, Zariski-dense subset $V_{D_1} \subset |L|$ such that for each $D_2 \in V_{D_1}$, $S(D_1,D_2)$ holds.''
\end{center}
Let us also fix some notation for Hirzebruch surfaces. 
\begin{defn}
Let $e \geq 0$ be an integer. Then we call 
\[
 \hir e = \Proj{ \sg{\Pn{1}} \oplus \sO{\Pn 1}{-e} }
\]
the $e$-th Hirzebruch surface. For the $0$-th Hirzebruch surface $\hir 0 \simeq \Pn{1} \times \Pn{1}$ we let $p_1$, $p_2$ denote the two projections to $\Pn{1}$ and set
\[
 \e = p_1^*\sO{\Pn 1}{1} \qquad \text{and} \qquad \f = p_2^*\sO{\Pn 1}{1}.
\]
For $e \geq 1$, let $\fun{p}{\hir e}{\Pn 1}$ denote the projection and set
\[
 \eInf = \sO{\hir e}{1} \qquad \text{and} \qquad \f = p^*\sO{\Pn 1}{1}.
\]
Note that, for $e \geq 1$, there is a unique curve $\CInf \subset \hir e$ with $\CInf^2 = -e$. This curve is smooth and rational and $\eInf = \sO{\hir e}{\CInf}$.
\end{defn}
The following elementary lemma is well-known. It is included here only because of its second statement which is maybe less well-known but equally elementary.
\begin{lem}
Let $a,b \in \Z$ with $a \geq b$ and set $E = \sO{\Pn 1}{a} \oplus \sO{\Pn 1}{b}$. \par 
Then $\Proj{E} \simeq \hir{a-b}$ and $\sO{\Proj E}{1} = \eInf \otimes a \cdot \f$.
\end{lem}
For the lack of a reference, the following result on the pull-back of ample line bundles by finite maps is included. 
\begin{lem}
Let $\fun f X Y$ be a finite morphism of compact complex spaces and let $L$ be an ample line bundle on $Y$. Then $f^*L$ is ample.
\end{lem}
\begin{proof}
 Recall Grauert's ampleness criterion \cite[Lemma following Satz 4]{ModUndExz}: a line bundle $M$ on $X$ is ample, if and only if for every irreducible closed analytic subset $Z \subset X$ of positive dimension, there exists $m \in \Z_{>0}$ and a section $s \in \Com{0}{Z}{M|_Z^{\otimes m}}$ that vanishes at some point of $Z$ but not on all of it. \par
So let $Z \subset X$ be an irreducible closed analytic subset with $\dim Z > 0$. Since $f$ is finite, we have $\dim f(Z) > 0$. By assumption $L$ is ample, so there exists $l \in \Z_{>0}$ and a section $s \in \Com{0}{f(Z)}{L|_{f(Z)}^{\otimes m}}$ that vanishes somewhere on $f(Z)$, but not everywhere. Hence $(f|_{Z})^*(s) \in \Com{0}{Z}{(f^*L)|_{Z}^{\otimes m}}$ vanishes somewhere on $Z$ but not everywhere. Thus, by Grauert's criterion, $f^*L$ is ample.    
\end{proof}
\begin{lem}\label{res2}
Let $X$ be a compact complex manifold of dimension $3$ and let $\Delta \subset X$ be an irreducible, reduced divisor. \par
Then there exists an embedded resolution of singularities
\begin{center}
 \begin{tikzpicture}[every node/.style={on grid}]
   \matrix (m) [matrix of math nodes, row sep = 3em, column sep = 3em, text height = 1.5ex, text depth = 0.25ex]
   {
    \hat \Delta & \hat X \\
    \Delta &  X \\ };
    \path[-stealth]
     (m-1-1) edge node [auto] {$h$} (m-2-1)
     (m-1-2) edge node [auto] {$g$} (m-2-2);
     \path[right hook->]
     (m-1-1) edge node [auto] {} (m-1-2)
     (m-2-1) edge node [auto] {} (m-2-2);
 \end{tikzpicture}
\end{center}
and an effective divisor $A$ on $\hat \Delta$ such that 
\[
 \dualizing{\hat \Delta} = h^*\dualizing{\Delta}(-A)
\]
and such that for each irreducible curve $B \subset \exc h$ with $\dim h(B) = 1$, we have 
\[
 B \subset \supp A.
\]
\end{lem}
\begin{proof}
This is, in fact, true for any $g$ that is a composition of blow-ups in smooth centers that are contained in the singular locus of (the strict transform of) $\Delta$. \par The proof is by induction on the number of blow-ups and easy, hence omitted.
\end{proof}
\begin{lem}\label{genOneToOne}

Let $S$ be a smooth surface, let $C_1,C_2$ be integral curves and let 
\[
 \fun g S {C_1}, \qquad \fun f S {C_2}
\]
be surjective holomorphic maps. Suppose that there is an integral curve $B_1 \subset S$ such that $f|_{B_1}$ is generically one-to-one and that there is an integral curve $B_2 \subset S$ such that for all but finitely many $p \in C_2$ there is an integral curve $C_p \subset f^{-1}(p)$ that intersects $B_2$ and that is contracted by $g$. \par
If $g|_{B_2}$ is generically one-to-one, then so is $f|_{B_2}$. 
\end{lem}
\begin{proof}
The fact that $f|_{B_1}$ is generically one-to-one implies that the general fibre of $f$ is connected. Since $S$ is smooth, the general fibre of $f$ is smooth. In particular, for all but finitely many $p \in C_2$ 
\[
 C_p = f^{-1}(p).
\]
Suppose that a general point $p \in C_2$ has more than one preimage point under the map $f|_{B_2}$. This means that $C_p$ intersects $B_2$ in more than one point. Since $g$ contracts $C_p$, there are two different points in $B_2$ that are mapped by $g$ to the same point. This shows that if $f|_{B_2}$ is not generically one-to-one, then $g|_{B_2}$ is also not.
\end{proof}
\begin{lem}\label{b3leqb1}
Let $S$ be a reduced, compact complex space of dimension $2$, let $C$ be reduced, compact complex space of dimension $1$ and let 
\[
 \fun f S C
\]
be a holomorphic map whose general fibre\footnote{That is, there is a finite set $F \subset C$ such that $f^{-1}(x) \simeq \Pn 1$ for all $x \in S \setminus F$.} is $\Pn 1$ and that has no fibres of dimension $2$. Assume that there exists a line bundle that is non-trivial on the general fibre of $f$. \par
Then the following inequality holds 
\[
 \betti 3 S \leq \betti 1 C.
\]
\end{lem}
\begin{rem}
If $S$ is projective, the last assumption is, of course, empty.
\end{rem}
\begin{proof}
The Leray spectral sequence implies that 
\[
 \betti 3 S \leq \sum_{p=0}^3 \com{3-p}{C}{\hdi{p}{f}{\Q_S}}.
\]
Since the topological dimension of $C$ is $2$ we have 
\[
 \Com{3}{C}{f_*\Q_S} = 0.
\]
The assumption that the general fibre of $f$ is $\Pn 1$ implies that the support of $\hdi{1}{f}{\Q_S}$ is a finite set, hence 
\[
 \Com{2}{C}{\hdi{1}{f}{\Q_S}} = 0.
\]
By the assumption that $f$ does not have any fibres of dimension $2$ we have 
\[
 \hdi{3}{f}{\Q_S} = 0.
\]
Summing up, we have
\begin{align}\label{ineq312}
\betti 3 S \leq \com{1}{C}{\hdi{2}{f}{\Q_S}}.
\end{align}
Let $L$ be a line bundle on $S$ that is non-trivial on a general fibre. Its first Chern class $c_1(L)$ induces an injective morphism of sheaves
\[
 0 \to \Q_C \to \hdi{2}{f}{\Q_S}.
\]
Since the general fibre of $f$ is irreducible, the cokernel of this morphism is a skyscraper sheaf on $C$. Hence we have
\[
 \com{1}{C}{\hdi{2}{f}{\Q_S}} \leq \com{1}{C}{\Q_C} = \betti 1 C. 
\]
Together with (\ref{ineq312}) the claim follows.
\end{proof}
The essential ideas of the following lemma can already be found in the proof of \cite[Lemma 4.7]{Nak}.
\begin{lem}\label{genBlowUp}
Let $X,Y$ be compact complex threefolds, let $\fun f X Y$ be a bimeromorphic holomorphic map, let $E \subset X$ be smooth, irreducible divisor with $E \not\subset \exc f$ and let $C \subset f(E) \cap f(\exc f)$ be an irreducible curve with $C \not\subset f(\exc{f|_{E}})$.
\begin{center}
\begin{tikzpicture}[every node/.style={on grid}]
   \matrix (m) [matrix of math nodes, row sep = 3em, column sep = 3em, text height = 1.5ex, text depth = 0.25ex]
   {
    D & E & X \\
    C & f(E) & Y \\ };
    \path[-stealth]
     (m-1-1) edge node [auto] {bir} (m-2-1)
     (m-1-2) edge node [auto]{bir} (m-2-2)
     (m-1-3) edge node [auto] {$f$} (m-2-3);
     \path[right hook->]
     (m-1-1) edge node [auto] {} (m-1-2)
     (m-1-2) edge node [auto] {} (m-1-3)
     (m-2-1) edge node [auto] {} (m-2-2)
     (m-2-2) edge node [auto] {} (m-2-3);
\end{tikzpicture}
\end{center}
Then there exists precisely one irreducible divisor $\Delta \subset \exc f$ satisfying 
\[
 (f|_{E})_*^{-1}(C) \subset \Delta.
\]
Furthermore, setting $D = (f|_{E})_*^{-1}(C)$, we have $\Delta \cap E \subset D \cup \exc{f|_{E}}$ and $\Delta$ and $E$ intersect transversally in a general point of $D$. \par 
If we additionally assume that
\[
 C \not\subset f(f^*H - \Delta - E)
\]
then, generically over $C$, $f$ is just the blow-up of $Y$ along $C$.  
\end{lem}
\begin{proof}
  Clearly we have the inclusion 
\[
 D = (f|_{E})_*^{-1}(C) \subset \exc{f}
\]
as $C \subset f(\exc f)$ and $f$ has connected fibres. So there is a divisor $\Delta$ satisfying $D \subset \Delta \subset \exc{f}$. We show that for any reduced $\Delta \subset \exc f$ such that each irreducible component of $\Delta$ contains $D$ we have
\begin{align}\label{resIdeal}
 \ideal{\Delta,X}|_{E} = \ideal{D,E}
\end{align}
outside of $\exc{f|_{E}}$.
This implies that $\Delta$ and $E$ intersect transversally in a general point of $D$. In particular, $\Delta$ is smooth in a general point of $D$, which implies that there can be only one integral divisor with $D \subset \Delta \subset \exc f$.\par
The inclusion $\subset$ of \eqref{resIdeal} is true on all of $E$ as $D \subset \Delta$ and $D$ is reduced. \par
Now let $x \in E \setminus \exc{f|_{E}}$. We want to show that the inclusion $\supset$ holds locally at $x$. So let $s \in (\ideal{D,E})_x$. By assumption $f|_{E}$ is an isomorphism near $x$ so that there is $t \in (\ideal{C,f(E)})_{f(x)}$ with $s = (f|_{E})^*(t)$. Let $u \in (\ideal{D,Y})_{f(x)}$ with $u|_{f(E)} = t$. Then we have $f^*u \in (\ideal{\Delta,X})_x$ as $\Delta$ is reduced and each component is mapped onto $C$. This show the claim as
\[
 s = (f|_{E})^*(t) = (f|_{E})^*(u|_{f(E)}) = (f^* u)|_{E} \in ({\ideal{\Delta,X}}|_E)_x. 
\]
If we assume additionally that $C \not\subset f(f^*H - \Delta - E)$, i.e.\ that $\mult{\Delta}{f^*H} = 1$ and that no divisor $\Delta \neq \Gamma \subset \exc f$ maps onto $C$, it easy to see that generically over $C$ we have
\[
 f^{-1}\ideal C  = \ideal{\Delta},
\]
which implies the last claim. 
\end{proof}
\begin{lem}\label{QBlowUp}
Let $0 \in U \subset \C^3$ be open, let $x,y,z$ be the coordinate functions of $\C^3$ and let 
\[
 \ideal{C_1} = (x,y) \cdot \sg U, \quad  \ideal{C_2} = (y,z) \cdot \sg U.
\]
Let furthermore $s_1, s_2 \in \sg{\C^3}(U)$ and $k_1,k_2 \geq 1$ be such that with $\ideal{D_i} = (s_i) \cdot \sg U$ we have
\begin{align*}
 (\ideal{D_1} + \ideal{D_2})|_{C_1 \setminus \{0\}} = (\ideal{C_1})|_{C_1 \setminus \{0\}}, \\
s \in \ideal{C_2}^{k_1} \setminus \ideal{C_2}^{k_1+1}, \qquad s_2 \in \ideal{C_2}^{k_2} \setminus \ideal{C_2}^{k_2+1}.
\end{align*}
Let $\fun{f}{\hat U}{U}$ denote the blow-up of $U$ along the ideal sheaf $\ideal{C_2}$. Let $\hat D_1$, $\hat D_2$ resp.\ $\hat C_1$ denote the strict transforms of $D_1$, $D_2$ resp.\ $C_1$. Let the $\sg{C_1}$-sheaves $Q_1,Q_2$ be defined by the following to exact sequences
\begin{center}
\begin{tikzpicture}[every node/.style={on grid}]
  \matrix (m) [matrix of math nodes, row sep = 1em, column sep = 2em, text height = 1.5ex, text depth = 0.25ex]
  {
   0 & (\ideal{D_1} + \ideal{D_2})|_{C_1} & \conormal{C_1}{U} & Q_1 & 0\\
   0 & (\ideal{\hat D_1} + \ideal{\hat D_2})|_{\hat C_1} & \conormal{\hat C_1}{\hat U} & Q_2 & 0\\
  };
  \path[-stealth]
   (m-1-1) edge node [auto] {} (m-1-2)
   (m-1-2) edge node [auto] {} (m-1-3)
   (m-1-3) edge node [auto] {} (m-1-4)
   (m-1-4) edge node [auto] {} (m-1-5)
   (m-2-1) edge node [auto] {} (m-2-2)
   (m-2-2) edge node [auto] {} (m-2-3)
   (m-2-3) edge node [auto] {} (m-2-4)
   (m-2-4) edge node [auto] {} (m-2-5);
\end{tikzpicture}
\end{center}
Then the following assertions hold:
\begin{itemize}
 \item Both $\supp{Q_1}$ and $\supp{Q_2}$ consist of one point and
 \item $\com{0}{C_1}{Q_1} = \com{0}{\hat C_1}{Q_2} + k_1 + k_2 - 1$.
\end{itemize}
\end{lem}
\begin{proof}
Elementary, hence omitted.
\end{proof}
The following proposition is well-known (cf.\ e.g.\ \cite[Proposition 3.A.7]{BarthelKaup}).
\begin{prop}\label{MayViet}
Let $\fun{f}{X}{Y}$ be a surjective morphism of compact complex spaces and let $B \subset Y$ be a closed complex subspace, such that with $A = f^{-1}(B)$ the restricted map $\fun{f|_{X \setminus A}}{X \setminus A}{Y \setminus B}$ is an isomorphism. \par
Then there exists an exact sequence
\[
 \begin{tikzcd}
   0 \arrow{r}{} & \Com{0}{Y}{\C} \arrow{r}{} & \Com{0}{X}{\C} \oplus \Com{0}{B}{\C} \arrow{r}{} \arrow[draw = none]{d}[name = Z, shape = coordinate]{} & \Com{0}{A}{\C}  \arrow[rounded corners, 
         to path={ -- ([xshift=2ex]\tikztostart.east)
                   |- (Z) [near end]\tikztonodes
                   -| ([xshift=-2ex]\tikztotarget.west)
                   -- (\tikztotarget)}]{dll}{} & \\
    & \Com{1}{Y}{\C} \arrow{r}{} & \Com{1}{X}{\C} \oplus \Com{1}{B}{\C} \arrow{r}{} & \Com{1}{A}{\C} \arrow{r}{} & \cdots
  \end{tikzcd}
\]
\end{prop}
\begin{lem}\label{b1leq2h1}
Let $C$ be an irreducible, compact complex space of dimension $1$.
Then 
\[
 \betti 1 C \leq 2\com{1}{C}{\sg C}.
\]
\end{lem}
\begin{proof}
Easy. Consider the normalization and use \cref{MayViet}.
\end{proof}
\begin{lem}\label{movingCurve}
Let $X$ be compact complex threefold, let $L$ be a line bundle on $X$ such that $\bs L$ does not have divisorial components and such that the map $\Phi$ induced by the linear system $|L|$ maps $X$ bimeromorphically to a smooth subvariety $Y \subset \Proj{\Com 0 X L}$:
\begin{center}
\begin{tikzpicture}[every node/.style={on grid}]
   \matrix (m) [matrix of math nodes, row sep = 3em, column sep = 3em]
   {
    X & Y \subset \Proj{\Com 0 X L} \\
   };
    \path[dashed,->]
     (m-1-1) edge node [auto] {$\Phi$} (m-1-2);
\end{tikzpicture}
\end{center}
%Then there exists a non-empty, Zariski-open subset $U \subset |L|$ and for each $D_1 \in U$ there exists a non-empty, Zariski-open subset $V_{D_1} \subset |L|$ such that for each $D_2 \in V_{D_1}$ 
Then for general $D_1,D_2 \in |L|$ the following statements hold.
\begin{enumerate}
\item Let $H_1,H_2 \in |\sO{\Pn n}{1}|$ denote the elements corresponding to $D_1,D_2$ via $\Phi$. Then $C_{D_1,D_2} = \Phi_*^{-1}(H_1 \cap H_2 \cap Y)$ can be defined and
\[
 D_1 \cap D_2 = C_{D_1,D_2} \cup \bs{L}.
\]
\item $L \cdot C_{D_1,D_2} \geq \deg Y$ and if equality holds then $C_{D_1,D_2} \cap \bs{L} = \emptyset$. 
\end{enumerate}
\end{lem}
\begin{proof}
For (1), just take $D_1,D_2 \in |L|$ general such that, for $H_1,H_2 \in |\sO{\Pn n}{1}|$ the corresponding hyperplanes, $H_1 \cap H_2$ does not intersect the $1$-dimensional locus where $\Phi^{-1}$ is not defined. \par
For the second part, consider a resolution of $\Phi$
\[
 \begin{tikzcd}[column sep = small]
  & Z \arrow[swap]{dl}{g} \arrow{dr}{f} & \\
  X \arrow[dashed]{rr}{\Phi}&  & Y
 \end{tikzcd} 
\]
such that $g(\exc g) \subset \bs L$. Note that we have an exceptional divisor $E$ on $Z$ such that $g^*L - E = f^* \sO{Y}{1}$ and $g(E) = \bs{L}$.
\begin{align*}
 L \cdot C_{D_1,D_2} = g^* L \cdot g_*^{-1}C_{D_1,D_2} &= (f^*\sO{Y}{1} + E) \cdot g_*^{-1}C_{D_1,D_2} \\&= \sO{Y}{1} \cdot (H_1 \cap H_2) + E \cdot g_*^{-1}C_{D_1,D_2} \\&= \deg Y + E \cdot g_*^{-1}C_{D_1,D_2}
\end{align*}
This certainly implies the inequality in (2) and since $g(E) = \bs L$, it also implies the second part of (2).
\end{proof}
The following lemma is crucial to the proof of \cref{projective}.
\begin{lem}\label{blowup-rest}
In addition to the assumptions of \cref{movingCurve} assume that all irreducible curves in $\bs L$ are smooth and that, in the notation of \cref{movingCurve}, 
\[
 L \cdot C_{D_1,D_2} = \deg{Y},
\]
for general $D_1,D_2 \in |L|$. \par
Then there is a resolution $\fun{g}{\tilde X}{X}$ of $\Phi$ 
\[
\begin{tikzpicture}
  \matrix (m) [matrix of math nodes, row sep=2.5em, column sep = 2.5em]
  { \tilde X & \\
    X & Y \\};
  \path[-stealth]
  (m-1-1) edge node [auto,swap] {$g$} (m-2-1)
  (m-1-1) edge node [auto] {$f$} (m-2-2);
  \path[dashed,->]
  (m-2-1) edge node [auto] {$\Phi$} (m-2-2) ; 
\end{tikzpicture}
\]
such that $g(\exc g) \subset \bs L$ and 
\[ 
E \geq F,
\]
where we let $E, F \geq 0$ denote the exceptional divisors on $\tilde X$ satisfying 
\[
 \can{\tilde X} = g^* \can{X} + F, \qquad f^* \sO{\Pn 3}{1} = g^*L - E.
\]
\end{lem}
\begin{proof}
By \cref{movingCurve}, the inequality $L \cdot  C_{D_1,D_2} \geq \deg Y$ always holds and equality implies that
\begin{align}\label{emptyintersection}
 \bs{L} \cap C_{D_1,D_2} = \emptyset.
\end{align}
I claim that \eqref{emptyintersection} implies that $\bs L$ has no isolated points. Indeed, let $x \in \bs L$. For general $D_1,D_2 \in |L|$, all integral curves 
\[
 C_{D_1,D_2} \neq B \subset D_1 \cap D_2
\]
are contained in $\bs L$ by \cref{movingCurve}. Since $x \in \bs L \subset D_1 \cap D_2$ and $D_1 \cap D_2$ is purely of dimension $1$, there exists an integral curve $x \in B \subset D_1 \cap D_2$ and since by \eqref{emptyintersection} $B \neq C_{D_1,D_2}$, we have $B \subset \bs{L}$, i.e.\ $x$ is not isolated in $\bs L$. \par 
Before we start with the resolution procedure, let us introduce some notation. Let $X_0 = X$ and $L_0 = L$. We call a sequence $(g_1,L_1), \ldots, (g_k,L_k), \ldots$ of blow-ups $\fun{g_k}{X_k}{X_{k-1}}$ along smooth centres with exceptional divisors $E_k = \exc{g_k}$ and line bundles $L_k$ on $X_k$ a partial resolution of $\bs L$, if for each $k$ we have
\[
 g_k(E_k) \subset \bs{L_{k-1}}
\] 
and 
\[
 L_k = g_k^*L_{k-1} - b_k E_k,
\]
where $b_k \geq 1$ is the maximal integer such that the injection 
\[
 \Com{0}{X_k}{g_k^*L_{k-1} - b_k E_k} \to \Com{0}{X_k}{g_k^*L_{k-1}}
\]
is onto. \par Let $h_k = g_1 \circ \cdots \circ g_k$ and let $F^k,E^k$ be effective divisors on $X_k$ such that 
\[
 \can{X_k} = h_k^* \can X + F^k, \qquad L_k = h_k^* L - E^k.
\] 
For an integer $k \geq 1$ and a line bundle $M$ on a compact complex manifold $Y$, set 
\[
 \sing{M,k} = \{y \in Y \colon s_y \in \operatorname{m}_y^k \cdot M_y, \forall s \in \Com 0 Y M \}.
\]
The principalization algorithm of \cite{BravoEncinasVillamayor} yields a sequence 
\[
 (X_0,L_0), \ldots, (X_{k_0},L_{k_0})
\]
that is a partial resolution of $\bs L$ such that for each $k$
\[
  g_k(\exc{g_k}) \subset \sing{L_k,2} \quad \text{and} \quad \sing{L_{k_0},2} = \emptyset.
\]
It is easy to see that the first condition implies that $E^k \geq F^k$, for all $k$. \par
In order to resolve the base locus completely, i.e.\ to find a partial resolution
\[
(X_{k_0+1},L_{k_0 + 1}), \ldots, (X_{k_1},L_{k_1})
\]
such that  $\sing{L_{k_1},1} = \emptyset$, while preserving the inequality $E^k \geq F^k$, we must be more carefully in the choice of the centres of the blow-ups. An easy inductive argument shows that $E^k \geq F^k$ is preserved, if we only blow-up curves or, if not, points that are contained in an $h_k$-exceptional divisor $\Gamma$ with 
\[
 \mult{\Gamma}{(E^k-F^k)} > 0.
\]
This can be achieved as follows: first blow-up singular points of irreducible curves contained in $\bs{L_k}$ as long as they exist. Note that, as $\sing{L_k,2} = \emptyset$, for each $x \in X$ there exists $D \in |L_k|$ such that $x \not\in \sing{D}$. Therefore, we do not create new singular curves in the base locus. After finitely many steps all irreducible curves in the base locus are smooth. \par
We continue by blowing-up these successively. It is clear that the number of curves in the base locus does not increase and that after finitely many steps the base locus does not contain any curves at all. As shown in the beginning of the proof it cannot contain isolated points. Hence it is actually empty. \par 
What remains to show is that for each singular, irreducible curve in $\bs{L_k}$ there is an integral divisor $\Gamma$ satisfying 
\[
 \mult{\Gamma}{(E^k - F^k)} > 0.
\]
This is done by induction. \par 
For $k=0$ all irreducible curves in the $\bs{L}$ are smooth by assumption so that the claim is true in this case. \par 
Now let $k > 0$ and let $C \subset \bs{L_k}$ be a singular, irreducible curve. \par 
First assume that $C \not\subset E_k = \exc{g_{k}}$ then $g_{k}(C)$ is singular (and irreducible) and the induction hypothesis yields an integral divisor $\tilde \Gamma \supset g_k(C)$ such that 
\[
 \mult{\tilde \Gamma}{(E^{k-1} - F^{k-1})} > 0.
\]
Then we have $\Gamma \deq (g_{k})_*^{-1}\tilde \Gamma \supset C$ and 
\[
 \mult{\Gamma}{(E^k - F^k)} = \mult{\tilde \Gamma}{(E^{k-1} - F^{k-1})} > 0.
\]
Now assume that $C \subset E_k = \exc{g_{k}}$. Then we have 
\[
 \dualizing{X_k} = g_{k}^*\dualizing{X_{k-1}}(a_k E_k),
\]
where $a_k = 1$ or $2$ depending on whether $g_{k}$ is the blow-up of a curve or of a point, and 
\[
 L_k = g_{k}^*L_{k-1} - b_k E_k
\]
where $b_k \geq 1$. In this notation 
\[
 \mult{E_k}{(E^k - F^k)} = b_k - a_k.
\]
The claim is that $b_k -a_k > 0$ so that we can take $\Gamma = E_k$. Suppose that $b_k \leq a_k$, i.e.\ that
\[
 (a_k,b_k) \in \{(1,1), (2,1), (2,2) \}.
\]
If $b_k=1$ we clearly get a contradiction to the fact that
\[
 C \subset E_k \cap \bs{L_k}
\]
is irreducible and singular. \par
Suppose that $a_k = b_k = 2$. In particular, $g_{k}$ is the blow-up of a point and $E_k \simeq \Pn 2$. For general $D \in |L_{k-1}|$ we have 
\[
 (g_{k})_*^{-1}D = g_{k}^*D - 2 E_k.
\] 
We have $C \subset \bs{L_k} \cap E_k \subset \hat D \cap E_k$ and
\[
 \hat D|_{E_k} \sim \sO{X_k}{-2E_k}|_{E_k} = \sO{E_k}{2},
\]
 hence $\hat D|_{E_k}$ is an (effective) curve of degree $2$ in $E_k \simeq \Pn 2$. Again, this contradicts the fact that $C$ is integral and singular.    
\end{proof}
\section{Results concerning global deformations}
In this section we prove some general results for global deformations of Fano manifolds of Picard rank $1$ and index $2$. Everything until \cref{bpf-ample} is well-known. \par
The proofs of the following two lemmas can also be found in \cite[(2.2)]{Pet89}. 
\begin{lem}\label{moish}
Let $\Delta \subset \C$ be the unit disc, let $\bigX$ be a complex manifold and let $\fun{\pi}{\bigX}{\Delta}$ be a proper, surjective submersion with connected fibres and assume that there exists a sequence $(s_n)_{n \in \N}$, $s_n \in \Delta$, such that $s_n \to 0$ and such that for all $n \in \N$ the fibre $\bigX_{s_n} = \pi^{-1}(s_n)$ is a Fano manifold of Picard rank $1$ and index $2$. \par
Then the following assertions hold.
\begin{itemize}
\item The line bundle $\dualizing{\bigX_0}^{-1} = \dualizing{\bigX}^{-1}|_{\bigX_0}$ is big.
\item The central fibre $\bigX_0$ is a Moishezon manifold.
\item For all $q > 0$, we have $\Com{q}{\bigX_0}{\sg{\bigX_0}} = 0$.
\end{itemize}
\end{lem} 
\begin{proof}
By adjunction $\dualizing{\bigX_t}^{-1} = \dualizing{\bigX}^{-1}|_{\bigX_t}$. The Semi Continuity Theorem implies that
\[
 \com{0}{\bigX_0}{\dualizing{\bigX_0}^{-m}} \geq \com{0}{\bigX_{s_n}}{\dualizing{\bigX_{s_n}}^{-m}} \geq C m^3,
\]
for some $C > 0$ that is independent of $n$. This shows that $\dualizing{\bigX_0}^{-1}$ is big, hence also that $\bigX_0$ is Moishezon. \par
In particular, the Hodge decomposition holds on $\bigX_0$. By Ehresmann's theorem we have $\Com{1}{\bigX_0}{\C} = 0$ and $\Com{2}{\bigX_0}{\C} = \C$. This together with the Hodge decomposition shows that 
\begin{align}\label{h1}
 \Com{1}{\bigX_0}{\sg{\bigX_0}} = \Com{2}{\bigX_0}{\sg{\bigX_0}} = 0.
\end{align}
Since $\com{0}{\bigX_0}{\dualizing{\bigX_0}^{-m}} \geq 2$ for some $m > 0$, we have $\Com{0}{\bigX_0}{\dualizing{\bigX_0}} = 0$, which implies $\Com{3}{\bigX_0}{\sg{\bigX_0}}$ by Serre duality.
\end{proof}
An easy consequence is the following theorem.
\begin{lem}\label{pic}
 Under the assumptions of \cref{moish}, possibly after shrinking $\Delta$, there is a line bundle $\bigL$ on $\bigX$ such that
\begin{align}\label{pic1}
 \Pic{\bigX_0} = \Z \bigL|_{\bigX_0} \qquad  \text{and} \qquad -\can{\bigX_0} = 2 \bigL|_{\bigX_0}.
\end{align}
Furthermore, $\bigL|_{\bigX_0}$ is big.
\end{lem}
\begin{proof}
 I first claim that, possibly after shrinking $\Delta$, 
\begin{align}\label{vanAfterShrink}
 \Com 1 {\bigX} {\sg {\bigX}} = \Com 2 {\bigX}{\sg {\bigX}} = 0.
\end{align}
Indeed, since $\pi$ is proper and $\Delta$ is Stein, we get 
\[
 \Com{q}{\Delta}{\hdi p {\pi} {\sg{\bigX}}} = 0, \quad \text{for} \quad q > 1.
\]
Therefore by the Leray spectral sequence we have $\Com{k}{\bigX}{\sg {\bigX}} = \Com{0}{\Delta}{\hdi k {\pi}{\sg{\bigX}}}$. By \cref{moish}, $\Com{k}{\bigX_0}{\sg{\bigX_0}} = 0$, for $k > 1$, and base change yields 
\[
 (\hdi k {\pi}{\sg{\bigX}})_0 = 0, \quad \text{for} \quad k>1,
\]
hence \eqref{vanAfterShrink} holds. \par
Therefore, considering the exponential sequences of $\bigX$ and $\bigX_t$, $t \in \Delta$, we get a commutative diagram
\[
\begin{tikzcd}
  0 \arrow{r} & \Com{1}{\bigX}{{\sg{\bigX}^{\times}}} \arrow{r} \arrow{d} & \Com{2}{\bigX}{\Z} \arrow{r} \arrow{d}{\simeq} & 0\\
  0 \arrow{r} & \Com{1}{\bigX_t}{\sg{\bigX_t}^{\times}} \arrow{r} & \Com{2}{\bigX_t}{\Z} \arrow{r} & 0,
\end{tikzcd}
\]
where the vertical map on the right is an isomorphism by Ehresmann's theorem. Thus also the vertical map on the left, which is just the restriction of line bundles, is an isomorphism. \par
Using this for $t=0$ and for one $t$ such that $\bigX_t$ is Fano of Picard rank $1$ and index $2$ yields (\ref{pic1}). As $-\can{\bigX_0} = 2 \bigL|_{\bigX_0}$ is big, so is  $\bigL|_{\bigX_0}$.
\end{proof}
\cref{moish,pic} show that under the assumptions of \cref{moish} $X = \bigX_0$, $L = \bigL|_{\bigX_0}$ satisfy the assumptions of the following lemma, whose first part is taken from \cite[Corollary 5.3.9, Corollary 5.3.10]{Kol}. 
\begin{lem}\label{div-curv} 
Let $X$ be a Moishezon threefold and $L \in \Pic X$ such that 
\[
 \Pic X = \Z L, \qquad -\can X = 2L, \qquad L \text{ big}.
\]
Then there is the following vanishing of cohomology groups
\begin{center}
\begin{tikzpicture}
 \matrix (m) [matrix of math nodes, row sep = 0.5em, column sep = 2em]
 {
 \Com{0}{X}{kL} = 0, \quad k<0, & \Com{3}{X}{kL} = 0, \quad k > -2 \\  
 \Com{1}{X}{kL} = 0, \quad k \leq 0, & \Com{2}{X}{kL} = 0, \quad k \geq -2\\
 };
\end{tikzpicture}
\end{center}
and 
\[ 
 \eChar{X}{kL} = 2 \cdot \frac{1}{6} k (k + 1)(k + 2) + k + 1.
\]
Furthermore, every divisor $D \in |L|$ is irreducible and reduced. If $D_1,D_2 \in |L|$ are distinct divisors, then the curve $C = D_1 \cdot D_2$ satisfies
\[
 \Com{0}{C}{\sg C} = \C, \quad \Com{1}{C}{\sg C} = \C, \qquad \dualizing C \simeq \sg C.
\]
\end{lem}
\begin{proof}
Note that the vanishing statements are column-wise equivalent by Serre duality. The first row is obvious, since $L$ is big. The vanishing 
\[
 \Com{2}{X}{\can X + aL} = 0, \quad \text{for} \quad a > 0 
\]
follows from \cite[Lemma 5.3.8]{Kol} and the remaining vanishing of $\Com{1}{X}{\sg X}$ follows from \cref{moish}. \par
We know the leading term of the cubic polynomial $\eChar{X}{kL}$ and we know its values for $k \in \{-2,-1,0\}$. This is enough to determine it completely.\par  
That every $D \in |L|$ is irreducible and reduced follows from $\Pic{X} = \Z L$. Let $D_1,D_2 \in |L|$ be distinct divisors and set $C = D_1 \cdot D_2$. By adjunction\[
 \dualizing C = \dualizing{D_1}(D_2|_{D_1})|_{C} = \dualizing{X}(D_1 + D_2)|_{C} = \sg C.
\]
 Using the above vanishing results, one easily deduces $\Com{0}{C}{\sg C} = 0$. Then $\Com{1}{C}{\sg C} = 0$ follows by duality.
\end{proof}
The following lemma by Nakamura is applied many times in this paper.
\begin{lem}\label{h1eq1}
Let $D_1,D_2 \in |L|$ be distinct and let $A \subset D_1 \cdot D_2$ a closed complex subspace. \par 
If $\com{1}{A}{\sg{A}} = 1$, then $A = D_1 \cdot D_2$.  
\end{lem}
\begin{proof}
This is \cite[Lemma 3.2]{Nak}. The essential point is that $\dualizing{D_1 \cdot D_2}$ is trivial.
\end{proof}
The following is the structure theorem of Kollár hinted at in the introduction.
\begin{lem}\cite[Theorem 5.3.12]{Kol}\label{kollarThm}
Under the assumptions of \cref{div-curv}, let $\Phi$ denote the map induced by $|L|$. \par
If $\com{0}{X}{L} \geq 4$ and if $L$ is not globally generated, then either 
\begin{itemize}
\item $\com{0}{X}{L} = 4$ and $\Phi$ maps $X$ bimeromorphically to $\Pn 3$ or
\item $\com{0}{X}{L} = 5$ and $\Phi$ maps $X$ bimeromorphically to a smooth quadric $\Qn 3 \subset \Pn 4$. 
\end{itemize}
\end{lem} 
If $X$ is a global deformation of Fano threefolds that are smooth intersections of two quadrics in $\Pn 5$, then $\com 0 X L \geq 6$ by semi continuity and \cref{kollarThm} shows that $L$ is globally generated. By \cref{bpf-ample}, $X$ is a Fano manifold, which proves \cref{easyProp} \par
If $X$ is a global deformation of quartic double solids, then $\com 0 X L \geq 4$ by semi continuity. Therefore we can still apply \cref{kollarThm} but a priori one of the two exceptional cases given there could occur. If we can exclude these two cases, $L$ must be globally generated. 
%The assertion of \cref{mainThm} is that the line bundle $L$ of the preceeding Lemma is ample. (And the classification of Fano threefolds implies that $L$ is basepointfree and that the induced morphism is a double covering of $\Pn{3}$, i.e. that $X_0$ is of type $V_2$). \par 
%Using repeatedly the fact that all divisors in $|L|$ are irreducible and reduced Kollár shows that -- in the situation given -- the map induced by $L$ has $3$-dimensional image and that it is either basepointfree or a bimeromorphic map to the three-dimensional projective space $\Pn{3}$ or a smooth quadric $\Qn{3}$. This suggests to split the proof into the following to lemmas
Then the following lemma shows that the assertion of \cref{mainThm} holds, i.e.\ that $-\can{X} = 2 L$ is ample.
\begin{lem}\label{bpf-ample}
 In addition to the assumptions of \cref{mainThm} assume that $\bigL|_{\bigX_0}$ is globally generated. \par
 Then $\bigL|_{\bigX_0}$ is ample.
\end{lem}
\begin{proof}
Let $X = \bigX_0$, let $L = \bigL|_{\bigX_0}$ and let $\Phi$ be the morphism induced by $|L|$. The aim is to show that $\Phi$ is finite. \par 
First we show that the image $Y = \Phi(X)$ has dimension $3$. In fact, Kollár shows this without assuming that $|L|$ is base point free, as a first step towards the proof of \cref{kollarThm}. When $L$ is assumed to be globally generated the proof is essentially trivial: \par 
Since $\com 0 {X} L \geq 2$, $Y$ is not a point. Suppose $Y$ is a curve. Since $\com 0 {X} L \geq 3$, it is not a line and there exists a hyperplane 
\[
 Y \not\subset H \subset \Proj{\Com 0 {X} L}
\]
intersecting $Y$ in more than one point. Then the divisor $\Phi^*H \in |L|$ is disconnected, which contradicts \cref{div-curv}. Suppose $Y$ is a surface. Since $\com 0 {X} L \geq 4$ there exist hyperplanes 
\[
 H_1,H_2 \subset \Proj{\Com 0 {X} L}
\]
such that $H_1 \cap H_2 \cap S$ is a finite set with more than one element. This shows that $C = \Phi^*H_1 \cap \Phi^*H_2$ is disconnected, which contradicts $\Com{0}{C}{\sg C} = \C$ (cf.\ \cref{div-curv}), hence $Y$ has dimension $3$. \par
Since $\Pic{X} \simeq \Z$, the morphism $\Phi$ cannot contract any divisor. Suppose that there exists $y \in Y$ such that 
\[
 \dim \Phi^{-1}(y) = 1.
\]
Since $-\can{X} = \Phi^*\sO{Y}{2}$, we have $\hdi{1}{\Phi}{\sg{X}} = 0$ by Grauert-Riemenschneider vanishing, in particular 
\begin{align}\label{H1ZERO}
 \Com{1}{\Phi^{-1}(y)}{\sg{\Phi^{-1}(y)}} = 0.
\end{align}
Now let $C \subset \Phi^{-1}(y)$ be an irreducible (and reduced) curve. By \eqref{H1ZERO}, $C$ is smooth and rational, which contradicts \cref{noSRC}.
\end{proof}
In order to proof \cref{noSRC}, we need the following result (cf.\ \cite[2.2 h)]{Pet89}).
\begin{lem}\label{notInfinite}
Let $X$ be a 3-fold and let $L$ be a big line bundle such that 
\[
 \Pic{X} \simeq \Z L.
\] 
Then there are only finitely many curves $C \subset X$ with $L \cdot C \leq 0$.   
\end{lem}
\begin{proof}
Let $\fun{f}{Y}{X}$ be a bimeromorphic morphism from a projective manifold $Y$ -- such $f$ exists as $X$ is Moishezon. Let $L$ on $Y$ be a very ample line bundle and let $A \in |L|$ be irreducible and reduced. Then $A \not\subset \exc{f}$ and $A = f(B)$ is an irreducible divisor, which is Cartier, since $X$ is smooth. There is an effective divisor $E$ on $Y$ such that
\[
 A = f^*B - E.
\]
Since $\Pic X = \Z L$, there is an integer $m \in \Z$ such that 
\[
 \sO{X}{B} = mL
\] 
and since $L$ is big, $m > 0$. \par
Let $C \subset X$ be an irreducible curve with $C \not\subset f(\exc{f})$. Then 
\[
 m L \cdot C = B \cdot C = f^*B \cdot f_*^{-1}C = (A + E) \cdot f_*^{-1}C > 0,
\]
hence $L \cdot C > 0$. As $\dim f(\exc{f}) = 1$, the claim follows.
\end{proof}
\begin{lem}\label{noSRC}
Under the assumption of \cref{mainThm} there does not exist a smooth rational curve $C \subset \bigX_0$ with $\can{\bigX_0} \cdot C = 0$.
\end{lem}
\begin{proof}
Assume to the contrary that there exists a  smooth rational curve $C \subset \bigX_0$ with $\can{\bigX_0} \cdot C = \can{\bigX} \cdot C = 0$. Then the following holds:
\begin{claim}
 The central fibre $\bigX_0$ contains infinitely many numerically trivial curves.
\end{claim}
Let $D$ be the connected component of the Douady space of compact complex subspaces of $\bigX$ that contains $[C]$ and let
\[
\begin{tikzcd}
S \arrow[hookrightarrow]{r} \arrow{d} & \bigX \times D \arrow{dl} \\
D &
\end{tikzcd}
\]
be the corresponding universal family. \par
For the local dimension of $D$ at $[C]$, the following inequality holds by \cite[Theorem 1.14]{Kol96}\footnote{The theorem is stated for schemes, but it also holds true for complex spaces (cf.\ \cite[Remark 1.17]{Kol96}).}
\begin{align}\label{locDim}
 \dim_{[C]} D \geq -\can{\bigX} \cdot C + (\dim \bigX - 3) = 1.
\end{align}
The intersection of the line bundle $\can{\bigX}$ with fibres of the flat morphism $S \to D$ with connected base is constant, i.e.\ for each $d \in D$,
\begin{align}\label{numtriv}
 \can{\bigX} \cdot S_d = \can{\bigX} \cdot C = \can{\bigX_0} \cdot C = 0.
\end{align}
 Consider the composite map 
\[
\begin{tikzcd}
 S \arrow[hookrightarrow]{r} &  \bigX \times D \arrow{r} & \bigX \arrow{r}{\pi} & \Delta. 
\end{tikzcd}
\]
This map is constant. Otherwise it would be open and since each compact complex subspace of $\bigX$ is contained in fibres of $\fun{\pi}{\bigX}{\Delta}$, this would imply that for each $t \in \Delta$ close to $0$, $\bigX_t$ contains a curve $S_{d(t)}$, for some $d(t) \in D$. By \eqref{numtriv} all these curves are numerically trivial, which contradicts the fact that $s_n \to 0$ and $\bigX_{s_n}$ is Fano for all $n$. \par
Hence all curves parametrized by $D$ are contained in $\bigX_0$, which together with \eqref{locDim} proves the claim. \par
The claim, however, contradicts \cref{notInfinite}, which proves the lemma.
\end{proof}
%To exclude that $X_0$ is mapped bimeromorphically to $\Qn{3}$, we follow the strategy given by Nakamura. That is, we study a resolution of the bimeromorphic map to compute/bound the third Betti number $\betti{3}{X_0}$ in order to get a contradiction to the equality
%\[
% \betti{3}{X_0} = 20 
%\]
%which holds by the Ehresmann Theorem because $\betti{3}{X_{s_n}} = 20$ for all $n$ (see e.g.\ \cite{Fano}).
What remains to show in order to finish the proof of \cref{mainThm} is that the two ``exceptional'' cases of \cref{kollarThm} cannot occur in our situation. So assume that, for some family $\bigX \to \Delta$ as in \cref{mainThm}, $X = \bigX_0$ falls into one of the exceptional cases of \cref{kollarThm}. Then by what we have already shown, one of the following two lemmas can be applied. \par
The proofs of these lemmas constitute the main part of the paper and are given in \cref{techLemma}.
\begin{lem}\label{B3SEQ4}
Let $X$ be a Moishezon threefold and let $L \in \Pic X$ be such that 
\[
 \Pic X = \Z L, \qquad -\can X = 2L, \qquad L^3 = 2, \qquad \com 0 X L = 4.
\]
Assume also that there is no smooth rational curve $C \subset X$ such that $L \cdot C = 0$ and that $L$ is not globally generated. \par 
Then
\[
 \betti 3 X \leq 12.
\]
\end{lem}
\begin{lem}\label{B3SEQ5}
Let $X$ be a Moishezon threefold and let $L \in \Pic X$ be such that 
\[
 \Pic X = \Z L, \qquad -\can X = 2L, \qquad L^3 = 2, \qquad \com 0 X L = 5.
\]
Then
\[
 \betti 3 X \leq 6.
\]
\end{lem} 
In any case, we get $\betti 3 {\bigX_0} \leq 12$. On the other hand, for a quartic double solid $V_2$ it is well-known that 
\[
\betti 3 {V_2} = 20
\]
(see e.g.~\cite{Fano} and the easy calculation can be found in my thesis (\cite{diss}).
Thus, by Ehresmann's theorem, we have 
\[ 
 \betti 3 {\bigX_0} = \betti 3 {V_2} = 20.
\]
This is a contradiction, which finishes the proof of \cref{mainThm} modulo the proofs of \cref{B3SEQ4,B3SEQ5}.
\section{An example}
The reader might wonder, whether the exceptions of \cref{kollarThm} can ocurr at all -- not necessarily as a global deformation of Fano manifolds. There are examples  of manifolds behaving this way (cf.\ \cite[Examples 5.3.14]{Kol}) and Kollár attributes them to Hironaka and Fujiki. In this section I give a new example, for which the line bundle $L$, in contrast to the previously known examples, has a reducible base locus. The precise claim is the following.
\begin{prop}
There exists a Moishezon $3$-fold $X$ satisfying 
\[
 \Pic{X} = \Z L, \qquad \com{0}{X}{L} = 4, \qquad L^3  = 2, \qquad -\can{X} = 2L,
\]
such that the map induced by $|L|$ is bimeromorphic.
\end{prop}
\begin{proof}
Start with $\Pn{3}$. Choose two distinct hyperplanes 
\[
 H_1,H_2 \subset \Pn{3}
\]
and write $C_0 \deq H_1 \cap H_2$ for the line of intersection. Furthermore choose pairwise distinct points $p_1,p_2,p_3,p_4 \in C_0$ und smooth plane cubics
\[
 C_1 \subset H_1, \qquad C_2 \subset H_2,
\]
such that
\[
 C_0 \cap C_1 = \{ p_1,p_2,p_3\}, \qquad C_0 \cap C_2 = \{p_1,p_2,p_4\}.
\]
Let $\fun{f_1}{Y}{\Pn{3}}$ denote the twisted blow-up of $\Pn{3}$ along $C_1$ and $C_2$, i.e.\ locally near $p_1$ we first blow-up $C_1$ and then the strict transform of $C_2$ and locally near $p_2$ we proceed in reversed order. \par 
For a hyperplane $H \in |\sO{\Pn{3}}{1}|$ let $\hat H = (f_1)_*^{-1}H$ denote its strict transform. Consider the restricted map
\[
 \fun{{f_1}|_{\hat H_1}}{\hat H_1}{H_1}.
\]
Outside of $C_2$ it is an isomorphism, since the smooth curve $C_1 \subset H_1$ is a Cartier divisor in $H_1$. Near $p_1$ it is also an isomorphism, since after blowing-up $C_1$ the strict transform of $C_2$ does not intersect the strict transform of $H_1$ locally over $p_1$. Near $p_2$ resp.\ near $p_4$ it is the blow-up in this point. Analogously the restricted map ${f_1}|_{\hat{H}_2}$ is the blow-up of $H_2$ in $p_1,p_3$. Let $A_2, A_4 \subset \hat H_1$ and $A_1,A_3 \subset \hat H_2$ be the corresponding exceptional curves. \par 
Let $\hat{C}_0$ be the strict transform of $C_0$ under $f_1$ and let $\Delta_1$, $\Delta_2$ be the two $f_1$-exceptional divisors mapping to $C_1$, $C_2$ respectively. Note that $\hat{C}_0$ is a smooth rational curve that intersects each of the divisors $\Delta_1$ and $\Delta_2$ transversally in exactly two distinct points. Denote these by $q_1,q_3$ resp.~$q_2,q_4$ in such a way that $f_1(q_i)= p_i$, for all $i = 1, \ldots, 4$. 
One also sees easily that  
\[
 \hat{C}_0 = \hat{H}_1 \cdot \hat{H}_2.
\]
Since $\hat{H}_i \in |{f_1}^* \sO{\Pn{3}}{1} - \Delta_i|$, the conormal bundle of $\hat C_0$ is
\begin{align}\label{con1}
 \conormal{\hat C_0}{Y} = (\ideal{\hat H_1} + \ideal{\hat H_2})|_{\hat C_0} \simeq \sO{\hat C_0}{1}^{\oplus 2}.
\end{align}
Now let
\[
 \fun{f_2}{\tilde X}{Y}
\]
be the blow-up of $Y$ along $\hat{C}_0$, let $E_0 = \exc{f_2}$, let $\tilde H_i = (f_2)_*^{-1} \hat H_i$ and let $f = f_1 \circ f_2$:
\[
 \begin{tikzcd}
  \tilde X \arrow{r}{f_2} \arrow[bend left]{rr}{f} & Y \arrow{r}{f_1} & \Pn 3.
 \end{tikzcd}
\]
The restricted map 
\[
 \fun{{f_2}|_{\tilde H_i}}{\tilde H_i}{\hat H_i}
\]
is an isomorphism for $i = 1,2$, since $\hat C_0 \subset \hat H_i$ is Cartier.
The two curves $l_1 = E_0 \cap \tilde H_1$ and $l_2 = E_0 \cap \tilde H_2$ are disjoint. There is an isomorphism $E_0 \simeq \Pn{1} \times \Pn{1}$, that identifies the first projection with the restricted map ${f_2}|_{E_0}$. Let $\f_0 \in \Pic{E_0}$ denote the class of the fibres with respect to this projection and $\e_0 \in \Pic{E_0}$ the class of those with respect to the other projection. Then from (\ref{con1}) it follows that
\[
 \conormal{E_0}{\tilde X} = \e_0 + \f_0.
\]
The contraction theorem of Fujiki and Nakano yields a complex manifold $\hat X$ and a holomorphic map 
\[
 \fun{g_2}{\tilde X}{\hat X},
\]
that contracts exactly the curves in $E_0$ with class $\e_0$. Let $E_i \deq g_2(\tilde H_i)$. The curves $l_1, l_2$ are contracted by $g_2$ (to distinct points). In particular we have $E_1 \cap E_2 = \emptyset$. The restricted map $\fun{{g_2}|_{\tilde H_i}}{\tilde H_i}{E_i}$ is the blow-down of the smooth rational curve $l_i \subset \tilde H_i$. Therefore we have $E_i \simeq \Pn{1} \times \Pn{1}$. We want to show that $E_1,E_2 \subset \hat X$ can be contracted. To this end, we compute their conormal bundles. On $\tilde X$ we have
\begin{align*}
 \conormal{\tilde H_1}{\tilde X}	&= \sO{\tilde X}{- \tilde H_1}|_{\tilde H_1} = (- f^*\sO{\Pn{3}}{1} + \Delta_1 + E_0)|_{\tilde H_1}= \\ 
					&= -(f|_{\tilde H_1})^*\sO{H_1}{1} + (f|_{\tilde H_1})^*\sO{H_1}{3} - A_2 + f|_{\tilde H_1}^*\sO{H_1}{1} - A_2 - A_4=\\
					&= 3 (f|_{\tilde H_1})^*\sO{H_1}{1} - 2A_2 - A_4 = ({g_2}|_{\tilde H_i})^*(\e_1 + 2\f_1),\\
 \conormal{\tilde H_2}{\tilde X} 	&= \sO{\tilde X}{- \tilde H_2}|_{\tilde H_2} = (- f^*\sO{\Pn{3}}{1} + \Delta_2 + E)|_{\tilde H_2} = \\
					&=- (f|_{\tilde H_2})^*\sO{H_2}{1} + (f|_{\tilde H_2})^*\sO{H_2}{3} - A_1 + (f|_{\tilde H_2})^*\sO{H_2}{1} - A_1 - A_3 \\
					&= 3 (f|_{\tilde H_2})^*\sO{H_2}{1} - 2A_1 - A_3 = ({g_2}|_{\tilde H_2})^*(\e_2 + 2\f_2).
\end{align*}
Since $g_2^* E_i = \tilde H_i$,  $(g_2|_{\tilde H_i})^* \conormal{E_i}{\hat X} = \conormal{\tilde H_i}{\tilde X}$. Hence we have $\conormal{\tilde H_i}{\tilde X} = \e_i + 2\f_i$, for $i = 1,2$, since $(g_2|_{\tilde H_i})^*$ is injective.  
So again by the contraction theorem of Fujiki and Nakano there is a complex manifold $X$ and a holomorphic map $\fun{g_1}{\hat X}{X}$ that contracts $E_1$ and $E_2$ along their second projections, i.e.\ such that curves with class $\f_i$ are contracted. These correspond to lines in $H_1$ through the point $p_2$ resp.\ lines in $H_2$ through the point $p_1$. One shows easily that the manifold $X$ satisfies 
\[
  \Pic{X} = \Z L, \qquad \com{0}{X}{L} = 4, \qquad L^3  = 2, \qquad \can{X} = -2L.
\]
The map induced by $|L|$ is just the natural bimeromorphic map that we get from the construction of $X$. If $X$ were projective, the line bundle $L$ would have to be ample, since $\Pic X = \Z L$ and $L$ is big. However, $L$ has degree $-1$ on the curve $B_1 = g_2(E_1)$, hence $L$ is not ample and $X$ is not projective. 
\end{proof}
\begin{rem}
In view of \cref{B3SEQ4} it is interesting to determine $\betti 3 X$. Since blowing-up a smooth rational curve does not change $b_3$ we have
\[
 \betti 3 X = \betti 3 {\hat X} = \betti 3 {\tilde X} = \betti 3 Y.
\]
By \cref{MayViet} we have an exact sequence
\[
 \begin{tikzcd}
 0 = \Com{3}{\Pn{3}}{\C} \arrow{r}{} & \Com{3}{Y}{\C} \oplus \Com{3}{C_1 \cup C_2}{\C} \arrow{r}{}\arrow[draw = none]{d}[name = Z, shape = coordinate]{} & \Com{3}{\Delta_1 \cup \Delta_2}{\C} \arrow[rounded corners, 
         to path={ -- ([xshift=2ex]\tikztostart.east)
                   |- (Z) [near end]\tikztonodes
                   -| ([xshift=-2ex]\tikztotarget.west)
                   -- (\tikztotarget)}]{dll}{} \\
  \Com{4}{\Pn 3}{\C} \arrow{r}{} & \Com{4}{Y}{\C} \oplus 0. &   
  \end{tikzcd}
\]
The mapping in the second line is injective and $\betti{3}{C_1 \cup C_2} = 0$ for dimensional reason, hence 
\[
 \betti 3 Y = \betti 3 {\Delta_1 \cup \Delta_2}.
\]
The first mapping of the following Mayer-Vietoris sequence 
\[
 \begin{tikzcd} 
 \phantom{0} & \Com{2}{\Delta_1}{\C} \oplus \Com{2}{\Delta_2}{\C} \arrow{r}{} \arrow[draw=none]{d}[name=Z,shape=coordinate]{} & \Com{2}{\Delta_1 \cap \Delta_2}{\C}
  \arrow[rounded corners, 
         to path={ -- ([xshift=2ex]\tikztostart.east)
                   |- (Z) [near end]\tikztonodes
                   -| ([xshift=-2ex]\tikztotarget.west)
                   -- (\tikztotarget)}]{dl}{} & \phantom{0}\\ 
   \phantom{0} & \Com 3 {\Delta_1 \cup \Delta_2}{\C} \arrow{r}{} & \Com{3}{\Delta_1}{\C} \oplus \Com{3}{\Delta_2}{\C} \arrow{r}{} & 0
 \end{tikzcd}
\]
%\begin{center}
%\begin{tikzpicture}[every node/.style={on grid}]
%  \matrix (m) [matrix of math nodes, row sep=1em, column sep = 1em, text height %= 1.5ex, text depth = 0.25ex]
%  {
%    \Com{2}{\Delta_1}{\C} \oplus \Com{2}{\Delta_2}{\C} & \Com{2}{\Delta_1 \cap \%Delta_2}{\C} & \phantom{0}\\ 
%    \Com 3 {\Delta_1 \cup \Delta_2}{\C} & \Com{3}{\Delta_1}{\C} \oplus \Com{3}{\Delta_2}{\C} & 0
%  \\};
%  \path[-stealth]
%  (m-1-1) edge node [auto] {} (m-1-2)
%  (m-1-2) edge node [auto] {} (m-1-3)
%  (m-2-1) edge node [auto] {} (m-2-2)
%  (m-2-2) edge node [auto] {} (m-2-3);
%  \end{tikzpicture}
%\end{center} 
is surjective, hence
\[
 \betti 3 {\Delta_1 \cup \Delta_2} = \betti 3 {\Delta_1} + \betti 3 {\Delta_2}.
\]
The surface $\Delta_1$ resp.\ $\Delta_2$ is the blow-up of a $\Pn 1$-bundle over $C_1$ resp.\ $C_2$ in two point. Therefore, by \cref{b3leqb1} we have 
\[
 \betti 3 {\Delta_i} = \betti 1 {C_i} = 2, \quad \text{for} \quad i = 1,2.
\]
Thus $\betti 3 X = 4$.
\end{rem}
\section{Proof of the main lemmas}\label{techLemma}
In this section the proofs of \cref{B3SEQ4,B3SEQ5} are given. Eventually we treat them separately, referring to the first as the ``projective space'' case and to the second as the ``quadric'' case. Before that we show some further results which are true in both cases.
\begin{lem}\label{base-loci}
Let $X$ be a Moishezon threefold an let $L \in \Pic X$ be such that 
\[
 \Pic X = \Z L, \qquad -\can X = 2L, \qquad L^3 = 2, \qquad \com 0 X L > 1.
\]
Let $D_1,D_2 \in |L|$ with $D_1 \neq D_2$ and set $C = D_1 \cdot D_2$. \par
Then 
\[
 \bs{L} = \bs{L|_{D_1}} = \bs{L|_C}.
\]
\end{lem}
\begin{proof}
The assertion is true if global sections of $L|_C$ lift to global sections of $L|_{D_1}$ and if these lift to global sections of $L$. Since, by \cref{div-curv},
\[
 \Com{1}{X}{\sg X} = \Com{1}{D_1}{\sg{D_1}} = 0,
\]
both liftings can, indeed, be performed.
\end{proof}
\begin{lem}\label{no-sings}
In addition to the assumptions of \cref{base-loci} assume that $\com 0 X L \geq 4$ and that $|L|$ is not base point free.\par
Then $\dim {\bs L} = 1$. Moreover, if $D_1,D_2 \in |L|$ with $D_1 \neq D_2$, then all irreducible curves $C_0 \subset D_1 \cap D_2$ are smooth and rational. 
\end{lem}
\begin{proof}
We first show that $\dim{\bs L} = 1$. Assume to the contrary that $\bs L$ does not contain any curves. Then the curve $C = D_1 \cdot D_2$ is irreducible (and reduced) for general $D_1,D_2 \in |L|$. Note that $\bs L \subset C$ and that
\begin{align}\label{LC2}
 L \cdot C = L^3 = 2.
\end{align}
I claim that $\bs L \subset \sing C$. \par 
If $C$ is smooth, this is immediate since then \eqref{LC2} together with $\dualizing C = \sg C$ (cf.\ \cref{div-curv}) shows that $L|_C$ is base point free, hence so is $L$ by \cref{base-loci}. \par 
Now we do not assume that $C$ is smooth, but let $p \in C$ be a smooth point. The aim is to show that $p \not\in \bs{L|_C} = \bs{L}$.
Let $\fun{\nu}{\tilde C}{C}$ denote the normalization. Then $\Com{0}{\tilde C}{\nu^*L^{-1}(p)} = 0$ since $\deg \nu^*L^{-1}(p) = -1$. Using Grothendieck-Serre Duality we get  
\[
 \Com{1}{C}{L(-p)} = \Com{0}{C}{L^{-1}(p)} \subset \Com{0}{\tilde C}{\nu^*L^{-1}(p)} = 0
\]
and this implies $p \not\in \bs{L|_C}$. Thus, indeed, $\bs L \subset \sing C$. \par 
Let $q \in \bs{L}$, let $\fun{g}{\tilde X}{X}$ be the blow-up of $X$ in $q$, let $E = \exc g$ and let $a \in \Z_{\geq 1}$ such that 
\[
 \tilde L = g^* L - a E
\]
has the same global sections as $L$ and such that $E \not\subset \bs{\tilde L}$. Let furthermore $\tilde C = g_*^{-1}C$ denote the strict transform. Since $q \in \sing{C}$, we have $E \cdot \tilde  C \geq 2$. Hence 
\[
 \tilde L \cdot \tilde C \leq L \cdot C - 2a = 2 - 2a \leq 0.
\]
This contradicts the second part of \cref{movingCurve}. Therefore $\dim \bs L = 1$.\par
For the second assertion of the lemma let $D_1,D_2 \in |L|$ with $D_1 \neq D_2$ and let $C_0 \subset D_1 \cap D_2$ be irreducible and reduced.\par
 Suppose that $\Com{1}{C_0}{\sg{C_0}} \neq 0$, then \cref{h1eq1} yields $D_1 \cdot D_2 = C_0$. If $C_0 \subset \bs{L}$, then for all $D_1,D_2 \in |L|$ we have $D_1 \cdot D_2 = C_0$, which is a contradiction. If $C_0 \not\subset \bs{L}$, then $\dim \bs L < 1$, which is also a contradiction. Therefore $\Com{1}{C_0}{\sg{C_0}} = 0$, which implies that $C_0 \simeq \Pn 1$.
\end{proof}
\begin{lem}\label{sMinusOne}
Under the assumptions of \cref{no-sings}, let $D_1,D_2 \in |L|$ be distinct divisors such that the curve $C = D_1 \cdot D_2$ is reduced. Let $A \subset D_1 \cap D_2$ be an irreducible component and let $B \subset D_1 \cap D_2$ be the union of the remaining components. Assume that $B \subset \bs{L}$ and that $B$ is connected. \par
Then 
\[
 L \cdot A = \com 0 X L - 1.
\]
\end{lem}
\begin{proof}
Since $B$ is reduced and connected (and non-empty by \cref{no-sings}), $\com 0 B {\sg B} = 1$ and by \cref{h1eq1} 
\begin{align}\label{h1AB0}
 \com 1 A {\sg A} = \com 1 B {\sg B} = 0.
\end{align}
Let the skyscraper sheaf $Q$ be defined via the exact sequence 
\[
 \begin{tikzcd}
  0 \arrow{r}{} & \sg C \arrow{r}{} & \sg A \oplus \sg B \arrow{r}{} & Q \arrow{r}{} & 0. 
 \end{tikzcd}
\]
The associated long exact sequence together with \eqref{h1AB0} yields 
\[
 \com 0 C Q =2.
\]
Tensorize the above short exact sequence with $L$ 
\[
 \begin{tikzcd}
   0 \arrow{r}{} & L|_{C} \arrow{r}{} & L|_A \oplus L|_B \arrow{r}{} & Q \arrow{r}{} & 0 
  \end{tikzcd}
\]
and look at the associated long exact sequence
\[
 \begin{tikzcd}
   0 \arrow{r}{} & \Com{0}{C}{L|_C} \arrow{r}{\alpha} & \Com{0}{A}{L|_A} \oplus \Com{0}{B}{L|_B} \arrow{r}{\beta} & \Com{0}{C}{Q} \simeq \C^2.
  \end{tikzcd}
\]
Since $B$ is reduced and contained in $\bs{L}$ and since, by \cref{base-loci}, $\bs L = \bs{L|_C}$, the image of $\alpha$ is contained in the first summand. Thus
\[ 
 \com{0}{A}{L|_A} \geq \com{0}{C}{L|_C} = \com 0 X L - 2 \geq 2.
\]
In particular, since $A \simeq \Pn 1$, the line bundle $L|_A$ is very ample, so that $\beta$ is surjective. Therefore, $\com{0}{B}{L|_B} = 0$ and 
\[ 
 \com{0}{A}{L|_A} = \com{0}{C}{L|_C} + 2= \com 0 X L,
\]
which, again since $A \simeq \Pn 1$, implies the claim. 
\end{proof}
\subsection{The quadric case}
The strategy of the proof of the following result is essentially identical with that of the proof of \cite[Theorem 4.2]{Nak}. Let us first restate concretely what we want to prove here.
\begin{lem}\label{quadric}
Let $X$ be a Moishezon threefold and let $L \in \Pic X$ be such that 
\[
 \Pic X = \Z L, \qquad -\can X = 2L, \qquad L^3 = 2, \qquad \com 0 X L = 5.
\]
Then $\betti 3 X \leq 6$.
\end{lem}
\begin{proof}
Let $\Phi$ denote the map induced by $|L|$. Of course, $\Phi$ is not a morphism. Therefore, by \cref{kollarThm}, it maps $X$ bimeromorphically to a smooth quadric $\Qn 3 \subset \Pn 4$:
\[
\begin{tikzcd}
X \arrow[dashrightarrow]{r}{\Phi} & \Qn 3 \subset \Pn 4
\end{tikzcd}
\]
 \par
Since $X$ and $\Qn 3$ have different index, there must exist a divisor 
\[
 Q \subset \Qn 3
\]
that is contracted by $\Phi^{-1}$. For general $H_1,H_2 \in |\sO{\Qn 3}{1}|$ the curve $H_1 \cap H_2$ intersects the divisor $Q$ in at least two points at which $\Phi^{-1}$ is defined and it satisfies 
\[
 H_1 \cap H_2 \not\subset Q.
\]
If $\Phi^{-1}$ contracts $Q$ to a point then the curve 
\[
 C_{D_1,D_2} = \Phi_*^{-1}(H_1 \cap H_2) \subset D_1 \cap D_2
\]
is singular, where we let $D_1,D_2 \in |L|$ denote the elements corresponding to $H_1,H_2$ via $\Phi$. This contradicts the second part of \cref{no-sings}. \par 
So $Q$ gets contracted to a curve, call it $B$, and the curve $C_{D_1,D_2}$ intersects $B$ in at least two distinct points. Therefore $\com{1}{C_{D_1,D_2} + B}{\sg{C_{D_1,D_2} + B}} \geq 1$ and by \cref{h1eq1} this implies that
\begin{align}\label{inters}
 D_1 \cdot D_2 = C_{D_1,D_2} + B.
\end{align}
Let $\fun{g}{\tilde X}{X}$ denote the blow-up of $X$ along $B$ and let $E \subset \tilde X$ denote the exceptional divisor. By (\ref{inters}), both  $D_1$ and $D_2$ are generically smooth along $B$, hence
\[
 \hat D_i \deq g_*^{-1}(D_i) = g^*D_i - E \in |g^*L - E|, \quad \text{for} \quad i=1,2.
\]
Again by (\ref{inters}), $D_1$ and $D_2$ intersect transversally at each point of $B \setminus C_{D_1,D_2}$, hence $\hat D_1 \cdot \hat D_2$ contains no component that is mapped onto $B$ by $g$. Since $C_{D_1,D_2}$ and $B$ intersect in two distinct points, we get, using $\com{0}{C_{D_1,D_2}}{Q_{D_1,D_2}} = 2$, that in each point of $C_{D_1,D_2} \cap B$ at least one of the divisors $D_1,D_2$ is smooth. This shows that $\hat D_1 \cdot \hat D_2$ does not contain any fibre of $g$, hence that
\[
 \hat D_1 \cdot \hat D_2 = g_*^{-1}(C_{D_1,D_2}).
\]
This implies that $\dim \bs{g^*L - E} < 1$ and that $\bs{g^* L - E} \subset g_*^{-1}(C_{D_1,D_2})$. Since $L \cdot C_{D_1,D_2} = 4$ by \cref{sMinusOne},
\[
 (g^*L - E) \cdot g_*^{-1}(C_{D_1,D_2}) = L \cdot C_{D_1,D_2} - E \cdot g_*^{-1}C_{D_1,D_2} = 4 - 2 = 2.
\]
So, in fact, $g^* L - E$ is base point free by \cref{movingCurve}. Let $\fun{f}{\tilde X}{\Qn 3}$ denote the induced morphism.
\[
 \begin{tikzcd}
  \tilde X \arrow{d}{g} \arrow{dr}{f}& \\
  X \arrow[dashrightarrow]{r}{\Phi} & \Qn 3
 \end{tikzcd}
\]
Since $\Pic X \simeq \Pic{\Qn 3}$, there exists a unique irreducible divisor $\Delta \subset \tilde X$ that is $f$-exceptional. For some $a \geq 1$, 
\begin{align*}
 g^*\can X + E = \can{\tilde X} &= f^* \can{\Qn 3} + a \Delta \\
 &\sim -3g^*L + 3 E + a \Delta,
\end{align*}
which implies that $a=1$ and $\Delta \sim g^* L - 2E$. Thus $f^*\sO{\Qn 3}{1} \sim \Delta + E$ and 
\begin{align}\label{QPB}
f^*Q = E + \Delta.
\end{align}
The image $f(\Delta)$, which is irreducible, cannot be a point. Otherwise the divisor $g(\Delta)$ would be contained in infinitely many divisors $D \in |L|$, which is absurd, since these are all irreducible by \cref{div-curv}. Therefore 
\[
 C_0 = f(\Delta)
\]
is an irreducible curve. From \cref{genBlowUp} it follows, using \eqref{QPB}, that outside of a finite set $f$ is just the blow-up of $\Qn 3$ along $C_0$. In particular outside of this set the fibre of $f|_{\Delta}$ is just $\Pn 1$ and each such fibre intersects $E$ in exactly one point. From \cref{b3leqb1} -- applied to the holomorphic map $f|_{\Delta}$ and the line bundle $\sO{\tilde X}{E}|_{\Delta}$ -- we therefore get 
\[
 \betti 3 \Delta \leq \betti 1 {C_0}.
\]
By \cref{MayViet} there is an exact sequence
\[
 0 = \Com{3}{\Qn 3}{\C} \to \Com{3}{\tilde X}{\C} \oplus \Com{3}{C_0}{\C} \to \Com{3}{\Delta}{\C},
\]
which shows that $\betti 3 {\tilde X} \leq \betti 3 {\Delta}$, since $\Com{3}{C_0}{\C} = 0$. Since $g$ is the blow-up of a smooth rational curve, $\betti 3 X = \betti{3}{\tilde X}$, hence
\[
 \betti 3 X  = \betti 3 {\tilde X} \leq \betti 3 {\Delta} \leq \betti 1 {C_0}.
\]
We bound $\betti 1 {C_0}$ by studying the map from its strict transform 
\[
 \tilde C_0 = (f|_E)_*^{-1}(C_0) \subset E.
\]
Note that $f(E) = Q$. Since $E$ is a Hirzebruch surface and $Q$ is a quadric, there are two cases: either $E = \Pn 1 \times \Pn 1$ or $E = \hir 2$. \par
Write $\conormal{B}{X} \simeq \sO{B}{a} \oplus \sO{B}{b}$, with $a \geq b$. Since by adjunction
\[
 a + b = \can X \cdot B - \deg \can B = - 2 L \cdot B + 2
\]
and since  $L \cdot B = L^3 - L \cdot C_{D_1,D_2} = -2$, we have $(a,b) = (4,2)$ or $(a,b) = (3,3)$, which implies
\[
 \Delta|_{E} \sim (g^*L - 2E)|_E = \begin{cases} 2 \e + 4 \f, \quad \text{if} \quad E \simeq \Pn 1 \times \Pn 1 \\ 2 \eInf + 6 \f, \quad \text{if} \quad E \simeq \hir 2. \end{cases}
\]
In both cases we have 
\[
 \ideal{\Delta, \tilde X}|_{E} = \ideal{\tilde C_0,E}
\]
\emph{generically} along $\tilde C_0$ by \cref{genBlowUp}. \par
First assume that $E \simeq \Pn 1 \times \Pn 1$. Then $f|_E$ is an isomorphism and $\tilde C_0 = \Delta|_{E}$. An easy calculation shows that $\com{1}{C_0}{\sg{C_0}} = 3$, hence $\betti{1}{C_0} \leq 2 \cdot 3 = 6$ by \cref{b1leq2h1}. \par 
Now assume that $E \simeq \hir 2$. Then 
\[
 2 \eInf + 6 \f \sim \Delta|_E = \tilde C_0 + a \CInf,
\] 
for some $a \geq 1$. Since $\tilde C_0$ is effective, $a \leq 2$. If $a = 2$, $\tilde C_0 \in |6f|$, which is a contradiction to the fact that $\tilde C_0$ is irreducible and reduced, hence $a = 1$. Thus $\f \cdot \tilde C_0 = 1$ so that $\tilde C_0$ is a smooth rational curve. Applying \cref{MayViet} to the morphism $f|_{\tilde C_0}$ yields the exact sequence
\[
 \begin{tikzcd}[column sep = 1em]
   0 \arrow{r}{} & \Com{0}{C_0}{\C} \arrow{r}{} & \Com{0}{\tilde C_0}{\C} \oplus \Com{0}{q}{\C} \arrow{r}{} & \Com{0}{\CInf \cap \tilde C_0}{\C} \arrow{r}{} & \Com{1}{C_0}{\C} \arrow{r}{} & 0
 \end{tikzcd}
\]
where $q \in Q$ is the center of the quadric cone $Q$. From this we get
\[
 \betti 1 {C_0} \leq \#(\CInf \cap \tilde C_0) - 2 + 1 \leq \eInf \cdot \tilde C_0 - 1 = 3. 
\]
Summing up, in both cases we have $\betti 3 X \leq \betti 1 {C_0} \leq 6$ as claimed. 
\end{proof}
\subsection{The projective space case}
Let us first state explicitly what we want to prove here.
\begin{lem}\label{projective}
Let $X$ be a compact complex threefold, let $L \in \Pic X$ be such that
\[
 \Pic X = \Z L, \qquad -\can X = 2L, \qquad L^3 = 2, \qquad \com 0 X L = 4.
\]
Assume in addition that the map induced by $|L|$ is bimeromorphic and that there does not exist any integral curve $C \subset X$ satisfying $L \cdot C = 0$. \par
Then $\betti 3 X \leq 12$.  
\end{lem}
Unfortunately, in this case the base locus cannot as easily be described as in the quadric case. If for general $D_1,D_2 \in |L|$ the the movable curve of $D_1 \cdot D_2$ intersects the $1$-dimensional locus of $\bs{L}$ in two different points, then the situation is still rather similar: $D_1 \cdot D_2$ is a reduced cycle of smooth, rational curves. If not, then the intersection of general $D_1,D_2 \in |L|$ is non-reduced. In the latter case we show that $X$ is a compactification of $\C^3$ by an irreducible divisor $\Delta_0 \subset X$ and that
\[
 \betti 3 X = \betti 3 {\Delta_0} = 0.
\]
Let us first make precise the notion of ``the movable curve of $D_1 \cdot D_2$'', for $D_1,D_2 \in |L|$ and state some simple properties of it.
\begin{lem}\label{leq2}
 Under the assumptions of \cref{projective}, let $D_1, D_2 \in |L|$ be general and let $H_1,H_2 \subset \Pn 3$ be the hyperplanes corresponding to $D_1, D_2$ via $\Phi$. \par
Then the curve $H_1 \cap H_2$ intersects the locus where $\Phi^{-1}$ is an isomorphism. In this case we define
\[
 C_{D_1,D_2} = \Phi^{-1}_* (H_1 \cap H_2)
\]
and call it the \textbf{movable curve} of $D_1 \cdot D_2$. \par
Note that $C_{D_1,D_2} \subset D_1 \cdot D_2$ and that we have (local) equality at a general point of $C_{D_1,D_2}$. We do not have equality per se and $C_{D_1,D_2}$ is a smooth rational curve. \par
For any irreducible, smooth rational curve $C \subset D_1 \cdot D_2$ along which $D_1 \cdot D_2$ is generically reduced, we define the skyscraper sheaf $Q_{D_1,D_2}^C$ via the following exact sequence 
\begin{align}\label{defQ}
 0 \to (\ideal{D_1} + \ideal{D_2})|_C \to \conormal{C}{X} \to Q_{D_1,D_2}^C \to 0.
\end{align}
Then
\[
 \com{0}{C}{Q_{D_1,D_2}^C} = 2.
\]
In the case that $C = C_{D_1,D_2}$, we write $Q_{D_1,D_2} = Q_{D_1,D_2}^{C_{D_1,D_2}}$.
\end{lem}
\begin{proof}
By \cref{no-sings}, $\dim{\bs L} = 1$, hence $C_{D_1,D_2} \neq D_1 \cdot D_2$. That $C_{D_1,D_2}$ is smooth and rational then follows from \cref{h1eq1}. \par 
To prove the second statement let $C \subset D_1 \cap D_2$ be an irreducible, smooth rational curve such that $D_1 \cdot D_2$ is generically reduced along $C$. Since $Q_{D_1,D_2}^C$ is a skysrcaper sheaf, $\com{1}{C}{Q_{D_1,D_2}^C} = 0$. Taking the Euler characteristic of \eqref{defQ} we get
\begin{align*}
 \com{0}{C}{Q_{D_1,D_2}^C} &= \eChar{C}{\conormal{C}{X}} - \eChar{C}{{\ideal{D_1}}|_C \oplus {\ideal{D_2}}|_C} \\
                       &= 2 - 2L \cdot C + 2 - ( - 2 L \cdot C + 2) = 2.
\end{align*}
To get the equality of the first and the second line we apply Riemann-Roch on $C$ and use that by adjunction  
\[
 \deg \conormal{C}{X} = - \deg \can C + \can X \cdot C = 2 - 2L \cdot C
\]
and that $\ideal{D_i} \simeq L^{-1}$.
\end{proof}
The lemma shows that, in particular, the skyscraper sheaf $Q_{D_1,D_2}$ is supported in at most two points. We split the proof of \cref{projective} into the two cases when the following statement holds resp.\ does not hold.
\begin{stmt}\label{twoPoints}
For general $D_1,D_2 \in |L|$ the cokernel $Q_{D_1,D_2}$ of the map
\[
 (\ideal{D_1} + \ideal{D_2})|_{C_{D_1,D_2}} \to \conormal{C_{D_1,D_2}}{X}
\]
is supported in two distinct points, where $C_{D_1,D_2}$ denotes the movable curve as defined in \cref{leq2}.
\end{stmt}
\vspace{2mm}
\paragraph{\sc First case}
We first assume that the statement holds.
\begin{assum}\label{firstcase}
In addition to the assumptions of \cref{projective} assume that \cref{twoPoints} does hold. 
\end{assum}
Then the base locus of $|L|$ can be determined very concretely as the following lemma shows. 
\begin{lem}\label{cycle}
Under \cref{firstcase} the $1$-dimensional part of $\bs L$ consists of an odd number $n$ of integral curves $B_1, \ldots, B_n$. For general $D_1,D_2 \in |L|$, their intersection decomposes as  
\[
 D_1 \cdot D_2 = C_{D_1,D_2} + \sum_{i=1}^nB_i
\]
and it is a cycle of smooth rational curves. Moreover we have
\[
 L \cdot C_{D_1,D_2} = 3 \qquad \text{and} \qquad L \cdot B_i = (-1)^i.
\]  
\end{lem}
\begin{rem}
Here a cycle of smooth rational curves is a reduced curve whose irreducible compononents are smooth rational curves and whose intersection graph is a cycle. The intersection graph of a reduced curve $C$ is the graph whose vertices are the irreducible components of $C$ and where two vertices are connected by an edge precisely when the corresponding irreducible components have non-empty intersection.
\end{rem}
\begin{proof}[Proof of \cref{cycle}]
The proof of the first part is that of \cite[Lemma 3.3]{Nak}.\par Choose $D_1,D_2 \in |L|$ such that $Q_{D_1,D_2}$ is supported in two different points, which is possible since we assume that \cref{twoPoints} holds. Then we can write $Q_{D_1,D_2} = \C_p \oplus \C_q$, for $p,q \in C_{D_1,D_2}$ with $p \neq q$. This shows that locally at $p$ at least one of the divisors $D_1,D_2$ is smooth and that
\[
 (\ideal{D_1} + \ideal{D_2})_p = (x, yz),
\] 
for some local coordinate functions $x,y,z$. Therefore there exists an irreducible curve 
\[
 p \in C_0 \subset D_1 \cdot D_2, \quad \text{with} \quad C_0 \neq C_{D_1,D_2},
\]
along which $D_1 \cdot D_2$ is generically reduced. Since $(Q_{D_1,D_2}^{C_0})_p = \C$, we can continue to choose curves in the same manner. As $D_1 \cap D_2$ contains only finitely many curves, this shows that there exists a cycle $C \subset D_1 \cdot D_2$ of smooth rational curves. Then, by \cref{h1eq1}, $C = D_1 \cdot D_2$, since $\com{1}{C}{\sg C} = 1$. \par
For $D_1,D_2 \in |L|$ general, such that additionally $D_1 \cap D_2 = C_{D_1,D_2} \cup \bs{L}$ (cf.\ \cref{movingCurve}), we get that
\[
 D_1 \cdot D_2 = C_{D_1,D_2} + \sum_1^n B_i
\]
for some $n$ and some smooth rational curves $B_1, \ldots, B_n \subset \bs L$. \par
By \cref{sMinusOne}, $L \cdot C_{D_1,D_2} = 3$, hence 
\begin{align}\label{minusOne}
 \sum_{i=1}^n L \cdot B_i = L^3 - L \cdot C_{D_1,D_2} = -1.
\end{align}
Let $C = D_1 \cdot D_2$. I claim that 
\begin{align}\label{van10}
 \Com{0}{C}{-L|_C} = 0,
\end{align}
Indeed, using the vanishing statements from \cref{div-curv}, we easily compute 
\[
 \Com{1}{C}{L|_C} = \Com{1}{D_1}{L|_{D_1}} = \Com{1}{X}{L}.
\]
Since, again by \cref{div-curv}, we have $\Com 2 X L = \Com 3 X L = 0$ and $\eChar X L = 4$, 
\[
 \com 1 X L = \eChar{X}{L} - \com 0 X L = 4 - 4 = 0
\] 
and by Grothendieck-Serre duality, $\Com{0}{C}{-L|_C} \simeq \Com{1}{C}{L|_C} = 0$, which proves \eqref{van10}. \par
 To study how $L$ restricts to the individual curves $B_i$, we prove the following claim.
\begin{claim}
Let $M$ be a line bundle on $C$.
\begin{itemize}
\item If there exists $i \in \{1, \ldots, n\}$, such that $M \cdot B_i \geq 2$. Then there exists a section $s \in \Com 0 C M$ satisfying $s|_{B_i} \neq 0$.
\item If there exists $i \in \{1, \ldots, n-1\}$, such that $M \cdot B_i = M \cdot B_{i+1} = 1$, then there exists $s \in \Com{0}{C}{M}$ such that $s|_{B_i + B_{i+1}} \neq 0$.
\end{itemize}
\end{claim}
For the first part, write $C = C^{\prime} + B_i$, set $S = C^{\prime} \cdot B_i$ and observe that $S$ consists of two reduced points. Since $B_i$ is smooth and rational,
\[
 \com{0}{B_i}{M|_{B_i}} = M \cdot B_i  + 1 \geq 3 > 2 = \com{0}{S}{M|_S}.
\]
The existence of the desired section now follows from the exact sequence
\[
\begin{tikzcd}
  0 \arrow{r} & \Com 0 C M \arrow{r} & \Com{0}{C^{\prime}}{M|_{C^{\prime}}} \oplus \Com{0}{B_i}{M|_{B_i}} \arrow{r} & \Com{0}{S}{M|_S} = \C^2.
 \end{tikzcd}
\]
The proof of the second part of the claim is completely analogous to that of the first part, after one observes that the assumption implies that 
\[
 \com{0}{B_i+B_{i+1}}{M|_{B_i+B_{i+1}}} \geq 3.
\]
The first part of the claim together with \eqref{van10} yields
\[
 -L \cdot B_i \leq 1, \quad \forall i \in \{1, \ldots, n\}
\]
Together with the fact that $B_i \subset \bs{L}$, for all $i$, it yields
\[
 L \cdot B_i \leq 1, \quad \forall i \in \{1, \ldots, n\},
\]
so that $L \cdot B_i \in \{-1,1\}$, for all $i$, since $L \cdot B_i = 0$ is excluded by the assumptions of \cref{projective}. The second part of the above claim then shows that 
\[
 L \cdot B_i + L \cdot B_{i+1} = 0, \quad \forall i \in \{1, \ldots, n-1\}.
\]
Since, by \eqref{minusOne}, $\sum_{i=1}^nB_i = -1$, this easily implies that $n$ is odd and that $L \cdot B_i = (-1)^i$.
\end{proof}
With this knowledge about the intersection of two general $D_1,D_2 \in |L|$ it is easy to show that the map $\Phi$ can be resolved by just two smooth blow-ups (in general with disconnected centres).
\begin{lem}\label{blowupCycle}
Under the same assumptions as in and in the notation of \cref{cycle}, let $\hat X \to X$ be the blow-up along the (disjoint) union 
\[
  \bigcup_{i \text{ odd}} B_i
\]
and let $\tilde X \to \hat X$ be the blow-up along the (disjoint) union
\[
  \bigcup_{i \text{ even}} \hat B_i,
\]
where $\hat B_i \subset \hat X$ is the strict transform of $B_i \subset X$, for even $i$. Let $\fun{g}{\tilde X}{X}$ denote the composition. \par
Then the meromorphic map $\Phi \circ g$ is a morphism and we call it $f$. Diagramatically:
\[
 \begin{tikzcd}[column sep = 3em]
  \tilde X \arrow{d} \arrow{ddr}{f} \arrow[bend right, swap]{dd}{g}& \\
  \hat X \arrow{d} & \\
  X \arrow[dashrightarrow]{r}{\Phi} & \Pn 3
 \end{tikzcd}
\]
\end{lem}
\begin{proof}
By \cref{cycle}, for general $D_1,D_2 \in |L|$, 
\[
 D_1 \cdot D_2 = C_{D_1,D_2} + B
\]
is a cycle of smooth rational curves. The strict transforms $g_*^{-1}D_1$, $g_*^{-1}D_2$ are certainly disjoint over the smooth points of $D_1 \cdot D_2$, hence their intersection contains only $g_*^{-1} C_{D_1,D_2}$ and possibly $g$-fibres of singular points of $D_1 \cdot D_2$. But since at each singular point of $D_1 \cdot D_2$ at least one of $D_1$, $D_2$ is smooth (cf.\ proof of \cref{cycle}), $g_*^{-1}D_1 \cdot g_*^{-1}D_2$ does not contain any fibre of $g$, hence  
\begin{align}\label{intIsStrict}
 g_*^{-1}D_1 \cdot g_*^{-1}D_2 = g_*^{-1} C_{D_1,D_2}.
\end{align}
Let $E \geq 0$ be the (reduced) sum of all the irreducible $g$-exceptional divisors. Then the line bundle
\[
 \tilde L = g^*L - E
\]
has the same global sections as $L$ and does not contain any divisor in its base locus. By \eqref{intIsStrict}, 
\begin{align}\label{baseInMoving}
 \bs{\tilde L} \subset g_*^{-1} C_{D_1,D_2}
\end{align}
 and since $C_{D_1,D_2} + B$ is a cycle of smooth rational curves, 
\[
 E \cdot g_*^{-1}C_{D_1,D_2} = 2. 
\]
Recalling that $L \cdot C_{D_1,D_2} = 3$ by \cref{sMinusOne}, we get  
\[
 \tilde L \cdot g_*^{-1}{C_{D_1,D_2}} = L \cdot {C_{D_1,D_2}} - E \cdot g_*^{-1}C_{D_1,D_2} = 3 - 2 = 1.
\]
By \cref{movingCurve} this implies that $\bs{\tilde L} \cap g_*^{-1}C_{D_1,D_2} = \emptyset$ and together with \eqref{baseInMoving} this yields $\bs{\tilde L} = \emptyset$.
\end{proof}
\begin{rem}
As $g$ is a composition of blow-ups in smooth rational curves we have 
\[
 \betti{3}{\tilde X} = \betti{3}{X}.
\]
\end{rem}
It turns out that the two cases when $\bs L$ contains only one curve or more than one curve  are essentially different. Therefore we make another case analysis.
\vspace{2mm}
\paragraph{\bf Irreducible base locus}
First assume that $\bs L = B_1$ is irreducible. Then $\Phi$ is resolved by just blowing-up $B_1$:
\[
 \begin{tikzcd}
  E \arrow[hookrightarrow]{r} \arrow{d} & \tilde X \arrow{d}{g} \arrow{dr}{f} & \\
  B_1 \arrow[hookrightarrow]{r}& X \arrow[dashrightarrow]{r}{\Phi} & \Pn 3
 \end{tikzcd}
\]
Let $E \subset \tilde X$ denote the irreducible exceptional divisor of $g$ and $\Delta \subset \tilde X$ the exceptional divisor of $f$, which is also irreducible as $\Pic X \simeq \Pic{\Pn 3}$. There exists an $a \in \Z_{\geq 1}$ such that 
\[
 g^* \can{X} + E = \can{\tilde X} = f^* \can{\Pn 3} + a \Delta.
\]
Using $f^* \sO{\Pn 3}{1} = g^* L - E$ and $\operatorname{gcd}(2,3) = 1$ we get  
\begin{align}\label{DeltaNum}
 \Delta \sim 2g^*L - 3 E.
\end{align}
The (irreducible) image
\[
 C_0 = f(\Delta)
\]
is not contained in any hyperplane, in particular it is of dimension $1$. We have 
\[
 C_0 \subset H_0 = f(E),
\]
since otherwise we would have $f^* H_0 = E$, which is absurd. \par
For general $D_1,D_2 \in |L|$, the movable curve $C_{D_1,D_2}$ intersects $B_1$ in exactly two points, hence for general $H_1,H_2 \in |\sO{\Pn 3}{1}|$, the line $H_1 \cdot H_2$ intersects $H_0$ in exactly two points. Therefore $H_0 \subset \Pn 3$ is a quadric and \eqref{DeltaNum} shows that 
\begin{align}\label{PBH}
 f^*H_0 = E + \Delta.
\end{align}
Since the map $f|_E$ is birational and $E$ a Hirzebruch surface, there are the two cases
\begin{enumerate}
\item $E = \hir 0$ and $\fun{f|_{E}} E {H_0}$ is an isomorphism, 
\item $E = \hir 2$, $\fun{f|_E} E {H_0}$ contracts $\CInf$ and $H_0$ is a quadric cone. 
\end{enumerate}
Applying the first part of \cref{genBlowUp} to
\begin{center}
\begin{tikzpicture}[every node/.style={on grid}]
   \matrix (m) [matrix of math nodes, row sep = 3em, column sep = 3em, text height = 1.5ex, text depth = 0.25ex]
   {
    (f|_E)_*^{-1}(C_0) & E & X \\
    C_0 & H_0 & Y \\ };
    \path[-stealth]
     (m-1-1) edge node [auto] {bir} (m-2-1)
     (m-1-2) edge node [auto]{bir} (m-2-2)
     (m-1-3) edge node [auto] {$f$} (m-2-3);
     \path[right hook->]
     (m-1-1) edge node [auto] {} (m-1-2)
     (m-1-2) edge node [auto] {} (m-1-3)
     (m-2-1) edge node [auto] {} (m-2-2)
     (m-2-2) edge node [auto] {} (m-2-3);
\end{tikzpicture}
\end{center}
yields that in a general point of $(f|_E)_*^{-1}(C_0)$, $\Delta|_E = (f|_E)_*^{-1}(C_0)$. Thus in the first of the above cases we have that $\Delta|_E$ is reduced and is mapped isomorphically to $C_0$. In the second case we have 
\begin{align}\label{resDelta0}
 \Delta|_E = (f|_E)_*^{-1}C_0 + \alpha \CInf,
\end{align}
for some $\alpha \geq 1$. Because of \eqref{PBH} the second part of \cref{genBlowUp} can also be applied showing that outside of a finite set, $f$ is the blow-up of $\Pn 3$ along $C_0$. In particular we can apply \cref{b3leqb1} to $\fun{f|_{\Delta}}{\Delta}{C_0}$ and $\sO{\tilde X}{E}|_{\Delta}$ to get 
\[
 \betti 3 {\Delta} \leq \betti 1 {C_0}.
\]
By \cref{cycle}, $g^* L|_{E} = (L \cdot B_1) \f = - \f$ and it is easy to compute that (cf.\ \eqref{conormal52} in the proof of \cref{simpleStructure})
\[
 \sO{\tilde X}{-E}|_E = \begin{cases} \e + 2 \f, \quad \text{if} \quad E \simeq \Pn 1 \times \Pn 1 \\
\eInf + 3 \f, \quad \text{if} \quad E \simeq \hir 2 \end{cases},
\]
hence by \eqref{DeltaNum}
\begin{align}\label{resDelta}
 {\Delta}|_E \sim (2g^* L - 3 E)|_E = \begin{cases} 3 \e + 4 \f, \quad \text{if} \quad E \simeq \Pn 1 \times \Pn 1 \\ 3 \eInf + 7 \f, \quad \text{if} \quad E \simeq \hir 2  \end{cases}.
\end{align}
In the first case one easily computes $\com 1 {C_0}{\sg{C_0}} = 6$, hence $\betti 1 {C_0} \leq 12$ by \cref{b1leq2h1}. \par 
In the second case, by \cref{MayViet}, there is an exact sequence
\[
\begin{tikzcd}
 0 \arrow{r} & \Com{0}{C_0}{\C} \arrow{r} & \Com{0}{\tilde C_0}{\C} \oplus \Com{0}{p}{\C} \arrow{r} \arrow[draw=none]{d}[name=Z,shape=coordinate]{} & \Com{0}{\tilde C_0 \cap \CInf}{\C} 
 \arrow[rounded corners, 
         to path={ -- ([xshift=2ex]\tikztostart.east)
                   |- (Z) [near end]\tikztonodes
                   -| ([xshift=-2ex]\tikztotarget.west)
                   -- (\tikztotarget)}]{dll}{} \\
  & \Com{1}{C_0}{\C} \arrow{r} & \Com{1}{\tilde C_0}{\C} \arrow{r} & 0,
\end{tikzcd}
\]
where $p \in H_0$ denotes the vertex of the quadric cone $H_0$ and $\tilde C_0 = (f|_E)_*^{-1}C_0$. Therefore,
\[
 \betti 1 {C_0} = \betti 1 {\tilde C_0} + \#(\CInf \cap \tilde C_0) - 1 \leq \betti 1 {\tilde C_0} + \CInf \cdot \tilde C_0 - 1.
\]
Putting together \eqref{resDelta0} and \eqref{resDelta}, we get that 
\[
 \tilde C_0 \sim (3-\alpha)\eInf + 7 \f,
\]
with $\alpha \in \{1,2\}$. If $\alpha = 2$, $\betti 1 {\tilde C_0} = 0$ and $\CInf \cdot \tilde{C_0} = 5$. If $\alpha = 1$, $\betti 1 {\tilde C_0} \leq 8$ and $\CInf \cdot \tilde{C_0} = 3$. In both cases 
\[
 \betti 1 {C_0} \leq 10 \leq 12.
\] 
This finishes the case when $\bs{L}$ contains only one curve.
\vspace{2mm}
\paragraph{\bf Reducible Base Locus} Assume now that $\bs L$ contains more than one curve. In this case we can also describe the geometry of $X$ fairly explicitly thanks to the following lemma.
\begin{lem}\label{simpleStructure}
In addition to \cref{firstcase} assume that $\bs{L}$ contains more than one curve. In the notation of \cref{cycle,blowupCycle}, let $E_i \subset \tilde X$ denote the unique divisor with $g(E_i) = B_i$. \par
Then $\bs{L}$ consists of three irreducible curves. The divisors $E_1,E_3$ are mapped birationally to two (distinct) hyperplanes $H_1,H_3 \subset \Pn 3$ by $f$, the divisor $E_2$ is mapped onto the line $l = H_1 \cap H_3$ by $f$ and
\[
  \deg_{H_1} f(\Delta_1) = \deg_{H_3} f(\Delta_3) = 3.
\]
\end{lem}
\begin{proof}
Let us first recall the set-up. By \cref{cycle}, the base locus of $L$ consists of an odd number of smooth rational curves $B_1, \ldots, B_n$, which together with the movable curve $C_{D_1,D_2}$ form a cycle, for general $D_1,D_2 \in |L|$. The resolution $g$ is constructed by first blowing-up the curves $B_i$ with odd index and then the strict transforms of those with even index. \par 
For $i = 1, \ldots,n$, let $E_i \subset \tilde X$ be the unique divisor with $g(E_i) = B_i$, let $\hat E_i = g_2(E_i)$, for odd index $i$, and let $\hat B_j = g_2(E_j)$, for even index $j$, as depicted in the following diagram:
\[
\begin{tikzcd}
    E_{i} \arrow[hookrightarrow]{r} \arrow{d} & \tilde X \arrow{d}{g_2} \arrow[bend right, near start, swap]{dd}{g} \arrow[hookleftarrow]{r} & E_{j}  \arrow{d} \\
    \hat E_{i} \arrow[hookrightarrow, crossing over]{r} \arrow{d} & \hat X \arrow{d}{g_1} \arrow[hookleftarrow]{r}& \hat B_{j} \arrow{d}\\
   B_{i} \arrow[hookrightarrow]{r} & X \arrow[hookleftarrow]{r} & B_{j}
\end{tikzcd}
\]
The assertion that there are only three exceptional divisors $E_1,E_2,E_3$ is shown in the end. First we prove that the ``edge'' divisors $E_1$ and $E_n$ are contracted to hyperplanes $H_1$ and $H_n$ and that the divisors $E_2,E_3, \ldots, E_{n-1}$ are contracted to the line $H_1 \cap H_n$. \par 
To do this, let us describe the geometry of $g$ in more detail beginning with the first blow-up. So let $1 \leq i \leq n$ be an odd number. Since $B_i$ is smooth and rational, its conormal bundle can be written as $\conormal{B_i}{X} \simeq \sO{B_i}{k_i} \oplus \sO{B_i}{l_i}$, where we assume that $k_i \geq l_i$. 
There is an injective map
\begin{center}
\begin{tikzpicture} 
  \matrix (m) [matrix of math nodes, row sep = 3em, column sep = 3em, text height = 1.5ex, text depth = 0.25ex]
  {
  0 & (\ideal{D_1} + \ideal{D_2})|_{B_i} & \conormal{B_i}{X} \\
  0 & (L^{-1} \oplus L^{-1})|_{B_i} & \sO{B_i}{k_i} \oplus \sO{B_i}{l_i} \\
  };
  \path[-stealth]
   (m-1-1) edge node [left] {} (m-1-2)
   (m-1-2) edge (m-1-3)
   (m-2-1) edge (m-2-2)
   (m-2-2) edge (m-2-3)
   (m-2-2) edge node [auto] {$\simeq$} (m-1-2)
   (m-2-3) edge node [auto] {$\simeq$} (m-1-3);
\end{tikzpicture}
\end{center}
which shows that $1 = - L \cdot B_i \leq l_i \leq k_i$. By the adjunction formula
\[
 2 = -2 L \cdot B_i = \deg \can X |_{B_i}= \deg \can{B_i} + \deg \conormal{B_i}{X} = - 2 + k_i + l_i,
\] 
hence 
\begin{align}\label{kl}
 (k_i,l_i) = (2,2) \qquad \text{or} \qquad (k_i,l_i) = (3,1).
\end{align}
In the first case
\[
 \hat E_i \simeq \Pn 1 \times \Pn 1 \quad \text{and} \quad \sO{\hat X}{- \hat E_i}|_{\hat E_i} = \e + 2 \f
\]
and in the second case
\[
 \hat E_i \simeq \hir 2 \quad \text{and} \quad \sO{\hat X}{- \hat E_i}|_{\hat E_i} = \eInf + 3 \f.
\]
Writing $\e = \eInf + 2 \f$ on $\hir 2$, the conormal bundles can be uniformly expressed as
\begin{align}\label{conormal52}
 \sO{\hat X}{- \hat E_i}|_{\hat E_i} = \e + l_i \f.
\end{align}
Note that we do not notationally indicate on which surface the line bundles $\e,\eInf,\f$ live. This is always clear from the context. \par
Now consider the second blow-up $\fun{g_2}{\tilde X}{\hat X}$, which creates the exceptional divisors with even index $E_2, E_4, \ldots, E_{n-1}$. The center of $g_2$ is the union of the curves $\hat B_j = (g_1)_*^{-1}B_j$ with even index $j$.
For general $\hat D_1,\hat D_2 \in |\hat L|$ these curves are isolated components of $\hat D_1 \cdot \hat D_2$, where 
\[
 \hat L = g_1^*L - \sum_{i \text{ odd}}\hat E_i
\]
is the line bundle that has the same global sections as $L$ and whose base locus does not contain any divisor. Therefore 
\[
 \conormal{\hat B_j}{\hat X} = (\ideal{\hat D_1} + \ideal{\hat D_2})|_{\hat B_j} \simeq (\hat L^{-1} \oplus \hat L^{-1})|_{\hat B_j}
\]
and since
\begin{align*}
 \hat L \cdot B_j &= (g_1^*L - \sum_{i \text{ odd}}{\hat E_i}) \cdot \hat B_j \\ &= L \cdot B_j- \hat E_{j-1} \cdot \hat B_j - \hat E_{j+1} \cdot B_j = 1 - 1 - 1 = -1,
\end{align*}
we get $\conormal{\hat B_j}{\hat X} = \sO{\hat B_j}{1} \oplus \sO{\hat B_j}{1}$, hence 
\begin{align}\label{evenEs}
E_j \simeq \Pn 1 \times \Pn 1 \qquad \text{and} \qquad \conormal{E_j}{\tilde X} = \e + \f. 
\end{align}
The divisors $E_1, E_n \subset X$ are the blow-up of $\hat E_1, \hat E_n \subset \hat X$ in the one-point sets $\hat E_1 \cap \hat B_2$, $\hat E_n \cap \hat B_{n-1}$ respectively. The divisor $E_i \subset X$, for odd $1 < i < n$, is the blow-up of $\hat E_i \subset \hat X$ in the two-point set $\hat E_i \cap (\hat B_{i-1} \cup \hat B_{i+1})$. \par
The movable curve $C_{D_1,D_2}$ (cf.\ \cref{leq2}) intersects each of $B_1$ and $B_n$ in precisely one point and each of these points is not contained in any other component of $\bs L$. Therefore a general line $l \subset \Pn 3$ intersects each of $f(E_1)$ and $f(E_n)$ in precisely one point, hence 
\[
H_1 = f(E_1) \qquad \text{and} \qquad H_n = f(E_n) 
\]
are hyperplanes. They are distinct, since $f$ is bimeromorphic. \par
On the other hand, the movable curve does not intersect any of the $B_2, \ldots, B_{n-1}$, hence $f$ has to contract the divisors $E_2, \ldots, E_{n-1}$. Therefore, since $\Pic X \simeq \Pic{\Pn 3}$, there exist exactly two irreducible $f$-exceptional divisors 
\[
\Delta_1,\Delta_n \subset \tilde X
\]
that are not $g$-exceptional. Since neither of $f^*H_1$, $f^*H_n$ can be $g$-exceptional, we can assume without loss of generality that $f(\Delta_i) \subset H_i$, for $i=1,n$ and we set 
\[
 C_1 = f(\Delta_1) \subset H_1, \qquad C_n = f(\Delta_n) \subset H_n.
\] 
Since for $i \in \{1,n\}$ there is at most one $D \in |L|$ that contains / is equal to $g(\Delta_i)$, there is at most one hyperplane $H \in |\sO{\Pn 3}{1}|$ that contains $C_i$. Therefore, $C_i$ is a curve and $\deg_{H_i}C_i \geq 2$, for $i \in \{1,n\}$. In fact, more can be said.
\begin{claim}\label{deg3}
 For $i \in \{1,n\}$, $\deg_{H_i} C_i \geq 3$. 
\end{claim}
\begin{proof}[Proof of the claim]
We have just seen that $\deg_{H_i}C_i \geq 2$. Suppose $\deg C_1 = 2$. Then there exists an irreducible quadric $Q \in |\sO{\Pn 3}{2}|$, such that $C_1 \subset Q$. This means that the pullback of $Q$ decomposes as 
\begin{align}\label{DECPBQ}
f^*Q = f_*^{-1}Q + a \Delta_1 + b \Delta_n + E^{\prime},
\end{align}
where $a \geq 1$, $b \geq 0$ and $E^{\prime} \in \sum_{i=2}^{n-1} \Z_{\geq 0}E_i$. Let $c \geq 0$ be such that $g(f_*^{-1}Q) \in |cL|$ and note that, in fact, $c \geq 1$ as $f_*^{-1}Q \not\subset \exc g$. \par
Pushing forward \eqref{DECPBQ} by $g$ yields 
\begin{align}\label{twoab}
 2 = c + a + b,
\end{align}
hence $a = c = 1$ and $b = 0$. In particular, 
\[
 g(f_*^{-1}(Q)) \in |L|.
\]
Let $H \in |\sO{\Pn 3}{1}|$ be the hyperplane corresponding to this divisor via $\Phi$. Then we have $Q \subset H$, which contradicts the fact that $Q$ is an irreducible quadric. Thus $\deg C_1 \geq 3$. By symmetry, $\deg C_n \geq 3$.
\end{proof}
We have seen that the divisors $E_2, \ldots, E_{n-1}$ get contracted by $f$. This can also be made more precise.
\begin{claim}
For all $2 \leq i \leq n-1$, $f(E_i) = H_1 \cap H_n$.
\end{claim}
\begin{proof}[Proof of the claim]
Let $y \in H_1 \cap H_n \setminus (C_1 \cup C_n)$. The aim is to show that then $y \in f(E_i)$, for all $2 \leq i \leq n-1$. First note that, since $E_1 \cap E_n = \emptyset$, 
\[
 H_1 \cap H_n \subset f(\exc f).
\]
Therefore we have 
\[
 f^{-1}(y) \subset \exc f \setminus ( \Delta_1 \cup \Delta_n), 
\]
so that 
\[
 f^{-1}(y) \subset \bigcup_{i=2}^{n-1}{E_i}.
\]
Observe that $f^{-1}(y) \cap E_1 \neq \emptyset$ and $f^{-1}(y) \cap E_n \neq \emptyset$ and that $\bigcup_{j=1,j \neq i}^n E_j$ is disconnected for all $2 \leq i \leq n-1$. Therefore 
\[
 f^{-1}(y) \cap E_i \neq \emptyset, \quad \text{for all} \quad 2 \leq i \leq n-1,
\]
as $f^{-1}(y)$ is connected. Since this is true for general $y \in H_1 \cap H_2$, this proves the claim.
\end{proof}
We now show that $n = 3$. To do this, let us examine how $\Delta_1$ and $\Delta_n$ restrict to the $E_i$'s. Observe that the proof of \cref{deg3} shows that the multiplicity of $f^* H_i$ along $\Delta_i$ is $1$. Therefore we can write
\begin{align}\label{pbacks}
 \Delta_1 =  f^* H_1 - E_1 - \sum_{i=2}^{n-1} a_i E_i, \qquad \Delta_n = f^*H_n - \sum_{i=2}^{n-1}b_i E_i - E_n,
\end{align}
where $a_i,b_i \geq 1$, because $f(E_i) \subset H_1 \cap H_n$, for $i=2,\ldots,n-1$. \par
Now observe that the effective divisor $\Delta_1|_{E_i}$ gets contracted by $f|_{E_i}$ for all $i > 1$. Otherwise we would have $f(\Delta_1) = f(E_i) = H_1 \cap H_n$, but we know that this is not the case. Thus, for all $i > 1$,
\begin{align}\label{LDEQ0}
 (\tilde L \cdot \Delta_1)_{E_i} = 0.
\end{align}
Recall that for odd $i$ the divisor $E_i$ is a blow-up of $\hat E_i$. Let us denote the exceptional curves of $g_2|_{E_i}$ in the following way:
\begin{align*}
F_1 &= E_2|_{E_1}, \qquad F_n = E_{n-1}|_{E_n}, \\
F_i &= E_{i-1}|_{E_i}, \quad G_i = E_{i+1}|_{E_i}, \quad \text{for} \quad 1 < i < n.
\end{align*}
Using \eqref{pbacks}, \eqref{conormal52} and the fact that $E_k \cap E_l = \emptyset$ if $|k-l| > 1$, we get, for odd indices $i$ and even indices $j$,
\begin{align*}
 \Delta_1|_{E_1} &\sim (\tilde L - E_1 - a_2E_2)|_{E_1} = \tilde L|_{E_1} + (g_2|_{E_1})^*(\e + l_i \f) - a_2 F_1, \\
 \Delta_1|_{E_i} &\sim (\tilde L - a_{i-1}E_{i-1} - a_i E_i - a_{i+1}E_{i+1})|_{E_i} \\ &= \tilde L|_{E_i} - a_{i-1}F_i + a_i(g_2|_{E_i})^*(\e + l_i \f) - a_{i+1}G_i\\ 
 \Delta_1|_{E_j} &\sim (\tilde L - a_{j-1}E_{j-1} - a_j E_j - a_{j+1}E_{j+1})|_{E_j} \\ &= \tilde L|_{E_j} - (a_{j-1} + a_{j+1})\f + a_j (\e + \f)\\
 \Delta_1|_{E_n} &\sim \tilde L|_{E_n} - a_{n-1}F_n.
\end{align*}
Together with (\ref{LDEQ0}) this yields
\begin{align*}
 0 &= (\Delta_1 \cdot \tilde L)_{E_i} = a_i - (a_{i-1} + a_{i+1}) \\
 0 &= (\Delta_1 \cdot \tilde L)_{E_j} = (k_j + l_j - 1)a_j - (a_{j-1} + a_{j+1}) = 3 a_j - (a_{j-1} + a_{j+1}) \\
 0 &= (\Delta_1 \cdot \tilde L)_{E_n} = 1 - a_{n-1}.
\end{align*}
This can be used to formally compute the $a_i$ successively:
\begin{align}\label{ais}
a_{n-1} &= 1, \quad a_{n-2} = 1, \quad a_{n-3} = 2, \\ \notag a_{n-4} &= 1, \quad a_{n-5} = 1, \quad a_{n-6} = 0.
\end{align}
Since $a_i \geq 1$, for all $i < n$, and $n$ is odd, this computation shows that either 
\[
 n = 3 \qquad  \text{or} \qquad n = 5.
\]
By \cref{genBlowUp}, the curve $\Delta_1|_{E_1}$ is generically reduced along $\tilde C_1 = (f|_{E_1})_*^{-1}(C_1)$, hence
\[
 \deg C_1 = \sO{H_1}{1} \cdot C_1 = ((f|_{E_1})^* \sO{H_1}{1} \cdot \Delta_1)_{E_1} = (\tilde L \cdot \Delta_1)_{E_1} = 4 - a_2.
\]
We already know from \cref{deg3} that this number is at least $3$. Since $a_2 \geq 1$, we have $a_2 = 1$. Then (\ref{ais}) implies $n = 3$, as claimed. \par
This also shows that $\deg C_1 = 3$. By symmetry, $\deg C_3 = 3$.
\end{proof}
We are now prepared to finish the proof of \cref{projective} under the assumption that \cref{twoPoints} holds. Since 
\[
 C_i \not\subset f(f^*H_i - E_i - \Delta_i) = f(E_2) = H_1 \cap H_3, \quad \text{for} \quad i \in \{1,3\},
\]
the second part of \cref{genBlowUp} shows that outside of $f(E_2)$ the morphism $f$ is the blowup of $\Pn 3$ along $C_1 \cup C_3$. \par
We now describe to some extend what happens over $f(E_2)$. 
\begin{claim}\label{contractionE2}
There exists a compact complex manifold $Y$ and there exist holomorphic maps $\fun{h}{\tilde X}{Y}$, $\fun{f_2}{Y}{\Pn 3}$ such that $h$ contracts precisely $E_2$ and such that we have a factorization $f = f_2 \circ h$:
\begin{center}
 \begin{tikzpicture}
  \matrix (m) [matrix of math nodes, row sep = 3em, column sep = 3em, text height = 1.5ex, text depth = 0.25ex]
  { 
   \tilde X & Y & h(\Delta_i)\\
   & \Pn 3 & C_i\\ 
  };  
  \path[-stealth]
   (m-1-1) edge node [auto] {$h$} (m-1-2)
   (m-1-1) edge node [auto, swap] {$f$} (m-2-2)
   (m-1-2) edge node [auto] {$f_2$} (m-2-2)
   (m-1-3) edge node [auto] {} (m-2-3);
  \path[left hook->]
   (m-1-3) edge node [auto] {} (m-1-2)
   (m-2-3) edge node [auto] {} (m-2-2);
 \end{tikzpicture}
\end{center}
\end{claim}
\begin{proof}[Proof of the claim]
By \eqref{evenEs}, $E_2 \simeq \Pn 1 \times \Pn 1$ and $\conormal{E_2}{\tilde X} = \e + \f$. Therefore, by the Fujiki-Nakano Contraction Theorem, there exists a holomorphic map $\fun{h}{\tilde X}{Y}$ that contracts precisely the curves lying in $|\e|_{E_2}$. Since $f$ contracts $E_2$ in the same direction, we get the factorization.
\end{proof}
The exceptional divisors of $f_2$ are precisely $h(\Delta_1)$ and $h(\Delta_3)$ and outside of a finite set $f_2$ is the blow-up along $C_1 \cup C_3$. In particular, for $i \in  \{1,3$\}, the general fibre of $f_2|_{\Delta_i}$ is $\Pn 1$ and it intersects $h(E_i)$ in (precisely) one point. Therefore, \cref{b3leqb1} can be applied to $f_2|_{h(\Delta_1) \cup h(\Delta_2)}$ and $\sO{Y}{h(E_1) + h(E_3)}$ yielding
\begin{align}\label{ineqf2}
 \betti 3 {\exc{f_2}} \leq \betti 1 {C_1 \cup C_3}.
\end{align}
\cref{MayViet} applied to $f_2$ gives an exact sequence
\begin{center}
\begin{tikzpicture}[every node/.style={on grid}]
  \matrix (m) [matrix of math nodes, row sep=2.5em, column sep = 2.5em, text height = 1.5ex, text depth = 0.25ex]
  { 
    0 = \Com{3}{\Pn{3}}{\C} & \Com{3}{Y}{\C} \oplus 0 & \Com{3}{\exc{f_2}}{\C},\\
  };
  \path[-stealth]
  (m-1-1) edge node [auto] {} (m-1-2)
  (m-1-2) edge node [auto] {} (m-1-3);
  \end{tikzpicture}
\end{center}
hence $\betti 3 Y \leq \betti 3 {\exc{f_2}}$. Using this inequality, the fact that $g$ and $h$ are (compositions of) blow-ups of smooth rational curves and (\ref{ineqf2}), we get 
\[
 \betti 3 X = \betti 3 {\tilde X} = \betti 3 Y \leq \betti 3 {\exc{f_2}} \leq \betti 1 {C_1 \cup C_3}.
\] 
The curves $C_1$, $C_3$ have degree $3$ in $H_1$, $H_3$ respectively. It is easy to see that this implies 
\[
 \betti 1 {C_1 \cup C_3} \leq 6.
\]
This finishes the first part of the proof of \cref{projective}. We now turn to the second case.
\vspace{2mm}
\paragraph{\sc Second case}
We are still proving \cref{projective}. In the case that \cref{twoPoints} holds, i.e.\ that for general divisors $D_1,D_2 \in |L|$ the moving curve $C_{D_1,D_2} \subset D_1 \cdot D_2$ intersects the remaining components of $D_1 \cdot D_2$ in two different points, we have just seen that the assertion $\betti 3 X \leq 12$ does hold. So from now on, in addition to the assumptions of \cref{projective}, we assume that \cref{twoPoints} does not hold.
\begin{assum}\label{projPlus}
Let $X$ be a compact complex threefold, let $L \in \Pic X$ satisfy
\[
 \Pic X = \Z L, \qquad -\can X = 2L, \qquad L^3 = 2, \qquad \com 0 X L = 4.
\]
Assume that the map induced by $|L|$ is bimeromorphic and that there does not exist any integral curve $C \subset X$ with $L \cdot C = 0$. \par
Additionally assume that \cref{twoPoints} does not hold.
\end{assum}
\begin{rem}
I do not know of any $X$ that satisfies \cref{projPlus}.
\end{rem}
\cref{projective} asserts that $\betti 3 X \leq 12$. We show that given \cref{projPlus}, in fact, $\betti 3 X = 0$. Recall from the introduction that we want to prove that $X$ is a compactification of $\C^3$ by an irreducible divisor $\Delta_0 \in |L|$. The first result towards this claim is the following lemma.\par
\begin{lem}\label{onlyOne}
Given \cref{projPlus}, let
\begin{center}
\begin{tikzpicture} 
  \matrix (m) [matrix of math nodes, row sep = 3em, column sep = 3em, text height = 1.5ex, text depth = 0.25ex]
  {
   X & \Pn 3 \\
  };
  \path[dashed,->]
   (m-1-1) edge node [auto] {$\Phi$} (m-1-2);
\end{tikzpicture}
\end{center}
be the map induced by the linear system $|L|$, which is bimeromorphic by assumption. \par 
Then $\Phi^{-1}$ contracts only one divisor and this divisor is a hyperplane.
% Let $H_0 \subset \Pn 3$ denote this hyperplane and let $\Delta_0 \subset X$ denote the unique divisor contracted by $\Phi$.
\end{lem}
\begin{proof}
Let $\fun{g}{\tilde X}{X}$ be a resolution of $\Phi$ satisfying $g(\exc g) \subset \bs L$ and let $f=\Phi \circ g$:
\begin{center}
 \begin{tikzpicture}
 \matrix (m) [matrix of math nodes, row sep = 3em, column sep = 3em]
 {
  \tilde X & \\
  X & \Pn 3\\
 };
 \path[-stealth]
  (m-1-1) edge node [auto,swap] {$g$} (m-2-1)
  (m-1-1) edge node [auto] {$f$} (m-2-2);
 \path[dashed,->]
  (m-2-1) edge node [auto] {$\Phi$} (m-2-2);
 \end{tikzpicture}
\end{center}
We first show that each divisor that is contracted by $\Phi^{-1}$ must be a hyperplane. Aiming for a contradiction, let $H_0 \subset \Pn 3$ be a divisor that is contracted by $\Phi^{-1}$ and suppose that $\deg H_0 \geq 2$. Let
\begin{align*}
 U = \left\{ 
       \begin{array}{ll}                & H_1 \cdot H_0 \text{ is reduced}, \\
         H_1 \in |\sO{\Pn 3}{1}| \colon &\dim(H_1 \cap f(\exc{f})) \leq  0 \quad \text{and} \\ 
                                        &\dim(H_1 \cap f(\exc{g})) \leq 1 
        \end{array}
     \right\}.
\end{align*}
This set is certainly Zariski-open and non-empty. Let $H_1 \in U$. Then the set 
\begin{align*}
 V_{H_1} = \left\{ 
            \begin{array}{ll}  
                                              & \#(H_0 \cap H_1 \cap H_2) \geq 2, \\ 
               H_2 \in |\sO{\Pn 3}{1}| \colon & H_1 \cap H_2 \cap f(\exc{f}) = \emptyset \quad \text{and} \\ 
                                              & \dim(H_1 \cap H_2 \cap f(\exc{g})) \leq 0
            \end{array}
           \right\}
\end{align*}
is also Zariski-open and non-empty. Let $H_2 \in V_{H_1}$. Then the line $l = H_1 \cap H_2$ satisfies
\[
 l \not\subset f(\exc{g}), \quad \#(l \cap H_0) \geq 2 \quad \text{and} \quad l \cap f(\exc{f}) = \emptyset.
\]
Therefore the strict/total transform $f_*^{-1}l = f^{-1}(l)$ satisfies 
\[
 \#(f^{-1}(l) \cap E_0) \geq 2, \qquad l \not\subset \exc{g},
\]
where $E_0 = f_*^{-1}H_0$.
Let $D_1,D_2 \in |L|$ denote the divisors corresponding to $H_1,H_2$ via $\Phi$. If two distinct points of $f^{-1}(l)$ get mapped by $g$ to the same point then this is a singular point of 
\[
 C_{D_1,D_2} \deq g(f^{-1}l) \subset D_1 \cap D_2.
\]  
This is not possible by \cref{no-sings}. Hence $g(E_0)$ is a curve and it intersects $C_{D_1,D_2}$ in two distinct points. Since 
\[
 g(E_0) \subset \bs{L} \subset D_1 \cap D_2,
\]
this implies that $Q_{D_1,D_2}$ is supported in two distinct points and since $D_1,D_2 \in |L|$ could be chosen general, \cref{twoPoints} holds; a contradiction to our assumption, so that, indeed, each divisor contracted by $\Phi^{-1}$ is a hyperplane. \par
The proof that $\Phi^{-1}$ contracts only one hyperplane is very similar. Suppose that $\Phi^{-1}$ contracts two distinct hyperplanes $H_1, H_2 \subset \Pn 3$. In analogy with the first part of the proof we can conclude that for general $H,H^{\prime} \in |\sO{\Pn 3}{1}|$ the strict transform of $l = H \cap H^{\prime}$ in $\tilde X$ is not contained in $\exc g$ and intersects each of 
\[
 E_1 = f_*^{-1}H_1, \qquad E_2 = f_*^{-1}H_2
\]
in one point and that that point is not contained $E_1 \cap E_2$. We can, however, not continue to conclude as before since $E_1 \cup E_2$ might be disconnected. We fix this by showing that $E_1$ can be connected to $E_2$ by $g$-exceptional divisors. The precise claim is that the set 
\begin{align}\label{connected5}
 A = E_1 \cup \bigcup_{\substack{E \subset \tilde X \text{ divisor,}\\ f(E) = l}}{E} \cup E_2
\end{align}
is connected and that it is $g$-exceptional. For connectedness, let $B,C \subset A$ be non-empty unions of irreducible components such that $A = B \cup C$. We have to show that $B \cap C \neq \emptyset$. Since for each irreducible component $E \subset A$ we have $l \subset f(E)$, we also have 
\begin{align}\label{contained5}
 l \subset f(B) \qquad \text{and} \qquad l \subset f(C).
\end{align}
Take $x \in l$ general such that $f^{-1}(x) \subset A$, then 
\[
f^{-1}(x) = (f^{-1}(x) \cap B) \cup (f^{-1}(x) \cap C).
\] 
By (\ref{contained5}) we have $f^{-1}(x) \cap B \neq \emptyset$ and $f^{-1}(x) \cap C \neq \emptyset$. Since $f^{-1}(x)$ is connected, this implies $B \cap C \neq \emptyset$.
\par
In order to show that all irreducible components of $A$ are $g$-exceptional, note that $E_1,E_2$ are $g$-exceptional and that $f(E) = l$ implies that $g(E)$ is contained in infinitely many divisors $D \in |L|$ which, of course, implies that $\dim g(E) < 2$ since the divisors in $|L|$ are irreducible.\par
Now, just as before, $f_*^{-1}(H_1 \cap H_2)$ intersects $A$ in at least two distinct points, for general $H_1,H_2 \in |\sO{\Pn 3}{1}|$. Since $C_{D_1,D_2} = g(f_*^{-1}(H_1 \cap H_2))$ is never singular, it intersects $g(A) \subset \bs L$ in two distinct points, hence $g(A)$ has dimension $1$ everywhere as it is connected. This shows that \cref{twoPoints} holds; a contradiction to our assumption, hence $\Phi^{-1}$ contracts at most one divisor and it is abvious that it contracts a least one.  
\end{proof}
\begin{lem}\label{DeltaToH}
Given \cref{projPlus}, let $\Delta_0 \subset X$ be the unique divisor contracted by $\Phi$ and let $H_0 \subset \Pn 3$ be the unique hyperplane contracted by $\Phi^{-1}$. \par
Then $\Delta_0 \in |L|$ and $\Phi_* \Delta_0 \subset H_0$.
\end{lem}
\begin{proof}
Let $\fun{g}{\tilde X}{X}$ be any resolution of $\Phi$ and let $\fun{f}{\tilde X}{\Pn 3}$ denote the morphism $\Phi \circ g$. Let furthermore $\Delta = g_*^{-1}\Delta_0$ and $E_0 = f_*^{-1}H_0$. 
\begin{center}
 \begin{tikzpicture}
 \matrix (m) [matrix of math nodes, row sep = 3em, column sep = 3em, text height = 1.5ex, text depth = 0.25ex]
 {
  \Delta & \tilde X & E_0 &\\
  \Delta_0  & X & \Pn 3 & H_0\\
 };
 \path[-stealth]
  (m-1-2) edge node [auto,swap] {$g$} (m-2-2)
  (m-1-1) edge node [auto,swap] {} (m-2-1)
  (m-1-3) edge node [auto,swap] {} (m-2-4)
  (m-1-2) edge node [auto] {$f$} (m-2-3);
 \path[dashed,->]
  (m-2-2) edge node [auto] {$\Phi$} (m-2-3);
 \path[right hook->]
  (m-1-1) edge node [auto] {} (m-1-2)
  (m-2-1) edge node [auto] {} (m-2-2);
 \path[left hook->]
  (m-1-3) edge node [auto] {} (m-1-2)
  (m-2-4) edge node [auto] {} (m-2-3);
 \end{tikzpicture}
\end{center}
Since $\Phi^{-1}$ contracts $H_0$, the divisor $E_0$ is $g$-exceptional. We can write
\begin{align}\label{pbH0}
 f^*H_0 = E_0 + a \Delta + F
\end{align}
where $a \geq 0$ and $F$ is an effective $f$-exceptional divisor that does not contain $\Delta$, hence is also $g$-exceptional. Let $d \geq 1$ be such that $\Delta_0 \in |dL|$. Then pushing forward (\ref{pbH0}) by $g$ yields the equation 
\[
 1 = a d,
\]
where we use that $f^*\sO{\Pn 3}{1} = g^*L - E$, for some effective $g$-exceptional divisor $E$. Therefore $a = 1$ and $d = 1$. The first equation shows $f(\Delta) \subset H_0$, i.e.\ $\Phi_* \Delta_0 \subset H_0$ and the second equation shows $\Delta_0 \in |L|$. 
\end{proof}
Let us pause to say something about how the proof continues: it is true, but we are not quite ready to proof that 
\begin{align}\label{b3Xb3DeltaPrelim}
 \betti 3 X = \betti{3}{\Delta_0}.
\end{align}
We begin to study the geometry of $\Delta_0$ aiming to bound $\betti 3 {\Delta_0}$. A first step is the construction of a nice embedded desingularization of $\Delta_0$ that is also a resolution of $\Phi$. In the course of this we are going to show that $\Phi_*^{-1}H_0$ is $1$-dimensional (cf.\ \cref{dimOne}). This can be used to show that $\Phi^{-1}$ is defined on $\Pn 3 \setminus H_0$ (cf.\ \cref{isCompactifi}), which then easily gives \eqref{b3Xb3DeltaPrelim}. The final step is to derive the inequality $\betti 3 {\Delta_0} \leq 1$ from properties of the desingularization.
\begin{lem}\label{C1}
In the notation of the previous lemma, let
\[
 C_1 = \Phi_* \Delta_0 \subset H_0. 
\]
Then $C_1$ is a curve and $\deg C_1 \geq 3$.
\end{lem}
\begin{proof}
Completely analogous to the proof of \cref{deg3} in \cref{simpleStructure} .
%The usual argument that $C_1$ is not contained in more that one hyperplane shows that $C_1$ is not a point and not a line. Suppose $C_1 \subset H_0$ has degree $2$. Then there exists an irreducible quadric $Q \subset \Pn 3$ such that $C_1 \subset Q$. 
\end{proof}
%\begin{rem}
%The following lemma shows that $B_1  = \Phi^{-1}_* H_0 \subset \Delta_0$ is also a curve. The proof is, however, much more involved than that for $C_1$.
%\end{rem}
 We now specify what kind of resolution of $\Delta_0$ we want to have.
\begin{lem}\label{niceResolution}
Under \cref{projPlus} and in the notation of \cref{C1}, the analytic subset $B_1  = \Phi^{-1}_* H_0 \subset \Delta_0$ is smooth rational curve. Moreover, there exists a commutative diagram 
\[
 \begin{tikzcd}\label{bigDiag}
  \hat \Delta \arrow{d}{} \arrow[bend right, swap]{dd}{\phi} \arrow[hookrightarrow]{r} & \hat X \arrow{d}{} \arrow{ddr}{f} \arrow[bend right, near start, swap]{dd}{g} & \\
  \Delta \arrow{d}{} \arrow[hookrightarrow, crossing over]{r} & \tilde X \arrow{d}{} \arrow{dr} & \\
  \Delta_0  \arrow[hookrightarrow]{r} & X \arrow[dashrightarrow]{r}{\Phi} & \Pn 3 
 \end{tikzcd}
\]
%\begin{center}
%\begin{tikzpicture} 
%  \matrix (m) [matrix of math nodes, row sep = 3em, column sep = 3em, text height = 1.5ex, text depth = 0.25ex]
%  {
%   \hat \Delta & \hat X & \\
%   \Delta  & \tilde X & \\
%   \Delta_0  & X & \Pn 3 \\ };
%   \path[-stealth]
%    %(m-1-1) edge node [bend left = 40] {$\hat g$} (m-3-1)
%    (m-1-1) edge node [left] {} (m-2-1)
%    (m-2-1) edge node [left] {} (m-3-1)
%    (m-1-2) edge node [right] {} (m-2-2)
%    (m-2-2) edge node [right] {} (m-3-2)
%    (m-1-2) edge node [auto] {$f$} (m-3-3)
%    (m-2-2) edge (m-3-3);
%   \path[right hook->]
%    (m-1-1) edge (m-1-2)
%    (m-2-1) edge (m-2-2)
%    (m-3-1) edge (m-3-2);
%   \path[dashed,->]
%    (m-3-2) edge node [auto] {$\Phi$} (m-3-3);
%\end{tikzpicture}
%\end{center}
with the following properties
\begin{enumerate}
 \item \label{prop1} $\hat \Delta$ is smooth. 
 \item \label{prop2}The meromorphic map $f = \Phi \circ g$ is a morphism.
 \item \label{prop3}There is a curve $B_1^{\prime} \subset \hat \Delta$ that is mapped isomorphically to the smooth rational curve $B_1$ by $g$ and that normalizes $C_1 \subset \Pn 3$ via $f$. 
 \item \label{prop4}There is at most one integral curve $B_1^{\prime \prime} \subset \hat \Delta$ with $B_1^{\prime \prime} \neq B_1^{\prime}$ and $g(B^{\prime\prime}) = B_1$.
 \item \label{prop5} Letting $\fun{\nu}{\tilde \Delta}{\Delta}$ denote the normalization and $\fun{\mu}{\hat \Delta}{\tilde \Delta}$ the naturally induced map, there is an effective divisor $A$ on $\hat \Delta$ such that
\[
 \dualizing{\hat \Delta} = \phi^*\dualizing{\Delta}(-A)
\]
and such that for each irreducible curve $C \subset \exc{\phi}$
\[
 C \subset \exc{\mu} \cup \supp{A}.
\]
\end{enumerate}
\end{lem}
The proof is rather lenghty, so we devote a whole subsection to it.
\subsection{Existence of a good resolution}
\begin{lem}\label{transversal}
Given \cref{projPlus}, let $D_1,D_2 \in |L|$ be distinct divisors and let $C,B \subset D_1 \cap D_2$ be distinct (reduced,) irreducible curves. \par
Then $C$ and $B$ intersect transversally.  
\end{lem}
\begin{proof}
Both $C$ and $B$ are smooth and rational by \cref{h1eq1}. Let $p \in C \cap B$. The claim is that
\[
 \ideal{C,p} + \ideal{B,p} = \operatorname{m}_p \subset \sg{X,p}.
\]
Suppose this is not the case, then $\com{1}{C+B}{\sg{C+B}} \geq 1$, hence by \cref{h1eq1}
\begin{align}\label{transAlongB}
 D_1 \cdot D_2 = C + B.
\end{align}
Without loss of generality $C \not\subset \bs L$. Then $B \subset \bs L$ and it is the only curve in $\bs L$. By \eqref{transAlongB}, $D_1 \cdot D_2$ is generically reduced along $B$ for some $D_1,D_2 \in |L|$, hence this is true for general $D_1,D_2 \in |L|$. Thus we have
\begin{align}\label{D1D2CB}
 D_1 \cdot D_2 = C_{D_1,D_2} + B,
\end{align}
for general $D_1,D_2 \in |L|$. By \cref{sMinusOne} this implies 
\[ 
 L \cdot C_{D_1,D_2} = 4 - 1 = 3.
\]
Using the second part of \cref{movingCurve}, it is easy to see that the blow up $\fun{g}{\tilde X}{X}$ along $B$ resolves all the indeterminacies of $\Phi$ (cf.\ proof of \cref{blowupCycle}). But this is absurd: letting $f = \Phi \circ g$ and $E = \exc{g}$, we have$f(E) = H_0$ so that $C_{D_1,D_2}$ and $B = g(E)$ intersect transversally in only one point, for general $D_1,D_2 \in |L|$, which, using \eqref{D1D2CB} contradicts $\com{1}{D_1 \cdot D_2}{\sg{D_1 \cdot D_2}} = 1$.
\end{proof}
\begin{lem}\label{dimOne}
Under \cref{projPlus} let $\fun{g}{\tilde X}{X}$ be any resolution of $\Phi$ 
\[
\begin{tikzpicture}
  \matrix (m) [matrix of math nodes, row sep=2.5em, column sep = 2.5em]
  { \tilde X & \\
    X & \Pn{3} \\};
  \path[-stealth]
  (m-1-1) edge node [auto] {$g$} (m-2-1)
  (m-1-1) edge node [auto] {$f$} (m-2-2);
  \path[dashed,->]
  (m-2-1) edge node [auto] {$\Phi$} (m-2-2) ; 
\end{tikzpicture}
\]
and let $E_0 = f_*^{-1} H_0$ and $\tilde C_1 = (f|_{E_0})_*^{-1}(C_1)$, where -- to remind the reader -- $H_0 \subset \Pn 3$ is the unique divisor contracted by $\Phi^{-1}$ and $C_1 = \Phi_*\Delta_0 \subset H_0$ is the image of the unique divisor $\Delta_0 \subset X$ contracted by $\Phi$. \par
 Then $\dim g(\tilde C_1) = 1$. In particular 
\[
 B_1 = \Phi_*^{-1}(H_0) = g(E_0)
\]
is $1$--dimensional. Furthermore $\ideal{\Delta_0} \not\subset \ideal{B_1}^3$.
\end{lem}
\begin{rem}
Note that the last two assertions of the lemma are independent of the resolution $g$.
\end{rem}
\begin{proof}
We need the following two claims whose elementary proofs are omitted here. Readers who do not want to work them out on their own can find them in the author's thesis \cite{diss} (Lemma 4.1.11 and beginning of the proof of Lemma 4.1.12 respectively). 
\begin{claim}\label{trans1}
 Let $C^{\prime} \subset \Delta$ be an integral curve with $f(C^{\prime}) = C_1$ and let $H \in |\sO{\Pn 3}{1}|$ be general. Then $\Delta$ and $f_*^{-1}H$ intersect generically transversally along each integral curve in $(f|_{\Delta})^{-1}(\Delta \cap H)$ that intersects $C^{\prime}$.
\end{claim}
\begin{claim}\label{clm2}
  There exist
\begin{itemize}
 \item irreducible divisors $E_1, \ldots, E_r \subset \exc g \cap \exc f$,
 \item integral curves $M_i \subset E_i \cap E_{i+1}$, for $ i \in \{0, \ldots, r-1 \}$, 
 \item an integral curve $M_r \subset E_r \cap \Delta$ and 
 \item integral curves $C_i^p \subset (f|_{E_i})^{-1}(p)$, for all $i \in \{1, \ldots, r\}$ and all but finitely many $p \in C_1$, such that
\begin{align*}
 C_1^p \cap M_0 \neq \emptyset, \qquad C_i^p \cap C_{i+1}^p \cap M_i \neq \emptyset, \qquad C_r^p \cap M_r \neq \emptyset.
\end{align*}  
\end{itemize}
In particular, $f(M_i) = C_1$, for all $i \in \{0, \ldots, r \}$, and $M_0 = \tilde C_1$.
\end{claim}
 Using these claims we want to show that $g$ does not contract $M_0 = \tilde C_1$. We, in fact, show that this is true for all the $M_i$.
First suppose that $g$ contracts $M_r$ and write
\[
 \{p\} = g(M_r).
\]
By \cref{C1}, $\deg{C_1} \geq 3$. Using additionally \cref{trans1} and the fact that $\Delta \not\subset  \exc g$ we see that for general $D \in |L|$ there exist (pairwise distinct) curves 
\[
 A_1, A_2, A_3 \subset D \cap \Delta_0, \quad \text{with} \quad p \in A_1 \cap A_2 \cap A_3
\]
and such that $D \cdot \Delta_0$ is generically reduced along each of these curves. Obviously, none of them is contained in $\bs L$. Therefore there is an integral curve 
\[
 B \subset \bs L \subset D \cap \Delta_0
\]
distinct from the $A_i$. By \cref{h1eq1}, $\com{1}{A}{\sg A} = 0$, for $A = A_1 + A_2 + A_3$, which shows that $A_1$, $A_2$, $A_3$ locally at $p$ look like the three coordinate axes. This in turn shows that $\com{0}{p}{Q_{\Delta_0, D}^{A_i}} \geq 2$ for all $i$ and by \cref{movingCurve} this implies that for every curve $B \subset \bs L$, 
\[
 B \cap A \subset \{p\}.
\]
Since by \cref{div-curv} $\Delta_0 \cap D$ is connected, there really is an integral curve 
\[
p \in B \subset \bs{L}.
\]
Then $\com{1}{A+B}{\sg{A+B}} = 1$, hence $\Delta_0 \cdot D = A + B$ which shows that $B$ is the only curve contained in $\bs L$ and that for general $D_1,D_2 \in |L|$ the intersection $D_1 \cdot D_2$ is reduced along $B$. As \cref{twoPoints} does not hold, there exist $D_1,D_2 \in |L|$ such that $D_1 \cdot D_2 = C_{D_1,D_2} + B$ and such that $C_{D_1,D_2} \cap B$ consists only of one point. Since $\com{1}{D_1 \cdot D_2}{\sg{D_1 \cdot D_2}} = 1$, $C_{D_1,D_2}$ and $B$ do not intersect transversally, which contradicts \cref{transversal}. Thus $g$ does not contract $M_r$. \par 
The following claim shows in particular that $g$ does not contract $M_0 = \tilde C_1$. \par
\begin{claim}\label{gContracts}
In the notation of \cref{clm2}, the curves $C_i^p$ get contracted by $g$ for all $i \in \{1,\ldots,r\}$ and almost all $p \in C_1$. Because $g$ does not contract $M_r$, it does not contract any $M_i$, for $i\in \{0, \ldots, r\}$.
\end{claim}
Suppose, to the contrary, that there exists $i \in \{1, \ldots, r\}$ such that for infinitely many $p \in C_1$, $\dim g(C_i^p) = 1$, and let $i_0$ be the maximal index having this property. Then, in fact, $g$ contracts $C_{i_0}^p$ for only finitely many $p \in C$. From the maximality of $i_0$ we get
\[
 g(E_i) = g(M_i) = g(M_r), \quad \text{for} \quad i > i_0.
\]
Since for general $p \in C_1$, $g(C_{i_0}^p) = g(M_r)$ and since $\deg{C_1} \geq 3$, we have
\begin{align}\label{inCube}
 \ideal{D} \subset \ideal{g(M_r)}^3,
\end{align}
for general, hence every, $D \in |L|$.
By \cref{trans1} and because $\Delta \not\subset \exc f$ there is $D \in |L|$ such that there is a reduced component $C \subset D \cdot \Delta_0$ such that  $C \cap g(M_r) \neq \emptyset$. Then by \eqref{inCube} and \cref{QBlowUp}, 
\[
 \com{0}{C}{Q_{\Delta_0,D}^C} \geq 3.
\]
A contradiction to \cref{movingCurve}, hence the first part of \cref{gContracts} follows. The second part is an immediate consequence.\par
 It remains to show that $\ideal{\Delta_0} \not\subset \ideal{B_1}^3$. Applying \cref{trans1} to $M_r \subset \Delta$ we see that there is a divisor $D \in |L|$ such that there exists a reduced component $C \subset D \cdot \Delta_0$ satisfying
\[
 C \cap B_1 \neq \emptyset.
\]
If $\ideal{\Delta_0} \subset \ideal{B_1}^3$, \cref{QBlowUp} yields $\com{0}{C}{Q^C_{\Delta_0,D}} \geq 3$. A contradiction to \cref{movingCurve}, hence $\ideal{\Delta_0} \not\subset \ideal{B_1}^3$.
\end{proof}
%Let $\fun{g_1}{X_1}{X}$ denote the blow-up along $B_1$, let $E_1 \subset X_1$ denote its exceptional divisor and set $L_1 = g^*L - E_1$. There is a curve $B_2 \subset \bs{L_1} \cap E_1$ that is mapped isomorphically to $B_1$. 
% We let $\fun{h_2}{X_2}{X_1}$ denote the blow-up of $B_2 \subset X_1$, $E_2 \subset X_2$ the exceptional divisor and $\fun{h_3}{\tilde X}{X}$ an arbitrary resolution of $\Phi \circ h_2 \circ g_1$
% \[
% \begin{tikzpicture}
%   \matrix (m) [matrix of math nodes, row sep=2.5em, column sep = 7em]
%   { \tilde X & \\
%     X_2 & \\
%     X_1 & \\
%     X & \Pn 3\\};
%   \path[-stealth]
%   (m-1-1) edge node [left] {$g_3$} (m-2-1)
%   (m-2-1) edge node [left] {$h_2$} (m-3-1)
%   (m-3-1) edge node [left] {$g_1$} (m-4-1)
%   (m-1-1) edge node [auto] {$f$} (m-4-2);
%   \path[dashed,->]
%   (m-4-1) edge node [auto] {$\Phi$} (m-4-2); 
% \end{tikzpicture}
% \]
%  We know that the strict transform $\tilde E_2 = (g_3)_*^{-1}E_2 \subset \tilde X$ is the unique $g$-exceptional divisor that is not contracted by $f$. Since $\tilde C_1 \subset \tilde E_2 \cap \exc f$ and $g(\tilde C_1) = B_1$, we have 
% \[
%  \tilde C_1 \subset ({g_3}_{|\tilde E_2})_*^{-1}(E_2 \cap {h_1}_*^{-1}(E_1 \cup \Delta_1)).
% \]
% Hence $g$ maps $\tilde C_1$ isomorphically to $B_1$.
Now the aim is to find two explicit blow-ups after which the induced bimeromorphic map to $\Pn 3$ is defined along the movable curve. After that, we apply \cref{blowup-rest} and \cref{res2} to get the desired resolution. \par 
The first step is to blow up $B_1$ and therefore we are interested in how the movable curve intersects $B_1$. There is the following result.
\begin{lem}\label{mib}
Suppose \cref{projPlus} holds. \par
Then for general $D_1,D_2 \in |L|$ the curve $C_{D_1,D_2}$ intersects $B_1$ transversally in exactly one point and $C_{D_1,D_2}$ and $B_1$ are the only irreducible component of $D_1 \cap D_2$ through this point. For all $D_1,D_2 \in |L|$ with this property the intersection $D_1 \cdot D_2$ is not generically reduced along $B_1$. In particular $\com{1}{D_1 \cap D_2}{\sg{D_1 \cap D_2}} = 0$.  
\end{lem}
\begin{proof}
Define the set  
\begin{align*}
 U =  \left\{ \begin{array}{rl}  & H_1 \neq H_0, \\ H_1 \in |\sO{\Pn 3}{1}| \colon &\dim(H_1 \cap f(\exc f)) \leq 0, \\ & \dim \Phi_*^{-1}(H_1 \cap H_0) = 1 \end{array} \right\}.
\end{align*}
Then $U$ is Zariski-open and non-empty. Let $H_1 \in U$. Then the sets 
\begin{align*}
 F_1 &= B_1 \cap \sing{\bs{L}} \quad \text{and}\\
 F_2 &= H_1 \cap f(\exc{f}) \cup \{q \in H_0 \cap H_1 \colon \Phi^{-1}(q) \in F_1\}
\end{align*}
are finite, hence
\[
 V_{H_1}  = \{H_2 \in |\sO{\Pn 3}{1}| \colon H_2 \cap F_2 = \emptyset, H_0 \cap H_1 \not\subset H_2 \}
\]
is Zariski-open and non-empty. Let $H_2 \in V_{H_1}$ and let $D_1,D_2 \in |L|$ be the divisors corresponding to $H_1,H_2$ via $\Phi$. Then the movable curve 
\[
 C_{D_1,D_2} = \Phi_*^{-1}(H_1 \cap H_2)
\]
is defined, intersects $B_1$ in exactly one point and this point is not contained in $\sing{\bs{L}}$, i.e.\ it is not contained in any irreducible component of $\bs{L}$ other than $B_1$. Since $H_1 \cap H_2 \cap f(\exc{f}) = \emptyset$, we have $D_1 \cap D_2 = C_{D_1,D_2} \cup \bs{L}$, hence $C_{D_1,D_2} \cap B_1$ is contained in no third component of $D_1 \cap D_2$. \par
By \cref{transversal}, $C_{D_1,D_2}$ and $B_1$ intersect transversally and by \cref{leq2}, 
\[
 \com{0}{C_{D_1,D_2}}{Q_{D_1,D_2}} = 2
\]
hence $D_1 \cdot D_2$ is not generically reduced along $B_1$. 
\end{proof}
\begin{lem}\label{moreGeneral}
Suppose \cref{projPlus} holds. \par
For general $D \in |L|$ we have $\ideal D \not\subset \ideal{B_1}^2$. Furthermore, for general $D_1,D_2 \in |L|$ in addition to the assertion of \cref{mib} we can achieve that locally around the unique point of $C_{D_1,D_2} \cap B_1$ there are coordinate functions $x,y,z$ such that locally
\[
 \ideal{D_1} = (x), \qquad \ideal{C_{D_1,D_2}} = (x,y), \qquad \ideal{B_1} = (x,z), \qquad \ideal{D_2} + \ideal{D_1} = (x,yz^2). 
\]
\end{lem}
\begin{proof}
Obviously the condition for $D \in |L|$ to satisfy $\ideal{D} \not\subset \ideal{B_1}^2$ is an open one. Suppose $\ideal{D} \subset \ideal{B_1}^2$ for all $D \in |L|$. Choose $D_1,D_2 \in |L|$ satisfying the assertions of \cref{mib}. Then we can apply \cref{QBlowUp} to get 
\[
 \com{0}{C_{D_1,D_2}}{Q_{D_1,D_2}} \geq 2+2-1 = 3.
\]
This contradicts \cref{leq2}, hence, $\ideal{D} \not\subset \ideal{B_1}^2$ for general $D \in |L|$. \par
We now show that for general $D_1,D_2 \in |L|$ the divisor $D_1$ is smooth at the unique point $p \in C_{D_1,D_2} \cap B$. The remaining assertions then follow easily. \par 
Let $U \subset |L|$ be the set from the proof of \cref{mib} and let 
\[
 U^{\prime} = \{ D \in U : \ideal{D} \not\subset \ideal{B_1}^2 \} = U \cap \{ D \in |L| : B_1 \not\subset \sing{D} \}.
\]
The set $U^{\prime}$ is Zariski-open and non-empty, by the first part of the proof. Let $D_1 \in U^{\prime}$ and let $H_1 \in |\sO{\Pn 3}{1}|$ be the corresponding hyperplane. Then the set $\sing{D_1} \cap B_1$ is finite, hence so are the sets
\begin{align*}
 F_1^{\prime} &= B_1 \cap (\sing{\bs L} \cup \sing{D_1}) \\
 F_2^{\prime} &= H_1 \cap f(\exc{f}) \cup \{q \in H_0 \cap H_1 \colon  \Phi^{-1}(q) \in F_1^{\prime} \}.
\end{align*} 
Therefore the set
\[
 V_{H_1}^{\prime} = \{ H_2 \in |\sO{\Pn 3}{1}| : H_2 \cap F_2^{\prime} = \emptyset, H_1 \cap H_0 \not\subset H_2 \}
\]
is  non-empty and Zariski-open. Let  $H_2 \in V_{H_1}^{\prime}$ and let $D_2 \in |L|$ be the divisor corresponding to it via $\Phi$. Then $D_1,D_2$ satisfy the assertions of \cref{mib} and in addition $C_{D_1,D_2} \cap B_1 \not\subset \sing{D_1}$. 
\end{proof}
After these preparations, we now begin to actually construct the resolution whose existence is claimed in \cref{niceResolution}. Let 
\[
 \fun{g_1}{X_1}{X}
\]
be the blow-up of $X$ along $B_1$, let $E_1 \subset X_1$ be its exceptional divisor and let 
\[
 L_1 = g_1^*L - E_1.
\]
Note that $L_1$ has the same global sections as $L$ and does not contain $E_1$ in its base locus since $B_1 \not\subset \sing{D}$ for general $D \in |L|$. \par 
For any $D \in |L|$ with $B_1 \not\subset \sing{D}$ there is unique irreducible curve 
\[
 B_D \subset E_1 \cap (g_1)_*^{-1}D \quad \text{with} \quad g_1(B_D) = B_1
\]
and $B_D$ is mapped isomorphically onto $B_1$. Let $D_1,D_2 \in |L|$ be general as in \cref{moreGeneral}. Then from the local description given by the lemma, we infer that $B_{D_1} = B_{D_2}$ locally, hence globally. Thus 
\[
 B_{D_1} \subset \hat D,
\]
for general, hence every, $\hat D \in |L_1|$. Setting $B_2 = B_{D_1}$, this just says that 
\[
 B_2 \subset \bs{L_1}.
\]
From the local description in \cref{moreGeneral} we also deduce that the divisors $\hat D_1 = (g_1)_*^{-1}D_1$ and $\hat D_2 = (g_1)_*^{-1}D_2$ intersect transversally in a general point of $B_2$. Furthermore it shows that
\begin{align}\label{CB2}
 C_{\hat D_1, \hat D_2} \cap B_2 \neq \emptyset
\end{align}
and that, writing $C = D_1 \cdot D_2 = C_{D_1,D_2} + B$, for some $B$, we have an exact sequence
\[
 \begin{tikzcd}
  0 \arrow{r} & \sg C \arrow{r} & \sg{C_{D_1,D_2}} \oplus \sg B \arrow{r} & \C^2 \arrow{r} & 0
 \end{tikzcd}
\]
Tensorizing this sequence with $L$ and taking global sections, yields the inequality $\com{0}{C_{D_1,D_2}}{L|_{C_{D_1,D_2}}} \leq 4$, which implies $L \cdot C_{D_1,D_2} \leq 3$, hence
\begin{align*}
 L_1 \cdot C_{\hat D_1,\hat D_2} =  (g_1^*L - E_1) \cdot C_{\hat D_1,\hat D_2} \leq 2.
\end{align*}
Now let 
\[
 \fun{g_2}{X_2}{X_1}
\]
be the blow-up of $X_1$ along $B_2$, let $E_2 \subset X_2$ be its exceptional divisor and let 
\[
 L_2 = g_2^*L_1 - E_2.
\]
Then $L_2$ has the same global sections as $L$ and $E_2 \not\subset \bs{L_2}$. By \eqref{CB2}, 
\begin{align}\label{CINTE2}
E_2 \cap (g_2)_*^{-1}C_{\hat D_1,\hat D_2} \neq \emptyset,
\end{align}
hence
\begin{align*}
 L_2 \cdot (g_2)_*^{-1} C_{\hat D_1,\hat D_2} &= (g_2^*L_1 - E_2) \cdot (g_2)_*^{-1} C_{\hat D_1,\hat D_2}  \\ &= L_1 \cdot C_{\hat D_1,\hat D_2} - E_2 \cdot (g_2)_*^{-1} C_{\hat D_1,\hat D_2} \leq 1.
\end{align*}
Then \cref{movingCurve} shows that, in fact, equality holds and that 
\begin{align}\label{DEFONC}
 \bs{L_2} \cap (g_2)_*^{-1}C_{\hat D_1,\hat D_2} = \emptyset.
\end{align}
Therefore \cref{blowup-rest} can be applied to $X_2$, $L_2$, yielding a morphism $\tilde X \to X_2$ that resolves $\bs{L_2}$. Let 
\[
 \begin{tikzcd}
  \tilde X \arrow{r} \arrow[bend left]{rrr}{\tilde g} & X_2 \arrow{r}{g_2} & X_1 \arrow{r}{g_1} & X,
 \end{tikzcd}
\]
be the composite map and let $f = \Phi \circ \tilde g$. Let furthermore $\tilde E,\tilde F \geq 0$ be the $\tilde g$-exceptional divisors on $\tilde X$ that satisfy
\[
  \can{\tilde X} = \tilde g^* \can{X} + \tilde F, \qquad f^* \sO{\Pn 3}{1} = g^*L - \tilde E.
\]
Using that the analogous result for $\tilde X \to X_2$ holds true by \cref{blowup-rest} and that $g_1$, $g_2$ are blow-ups along curves, we see that $\tilde E \geq \tilde F$. Thus the following lemma can be applied to $\tilde g$.
\begin{lem}\label{res1}
 Given \cref{projPlus}, let $\fun{g}{\tilde X}{X}$ be a resolution of $\Phi$, let $E,F,G \geq 0$ the divisors satisfying
\[
  \can{\tilde X} = g^* \can{X} + F, \qquad f^* \sO{\Pn 3}{1} = g^*L - E, \qquad \Delta = g^*\Delta_0 - G,
\]  
where $\Delta_0 \in |L|$ is the unique divisor contracted by $\Phi$ and $\Delta = g_*^{-1}(\Delta_0) \subset \tilde X$.
If $E \geq F$, then the following two statements hold:
\begin{itemize}
 \item $G \geq F$ and $\supp{G - F} = \exc{g}$.
 \item There is an effective Cartier divisor $A \geq 0$ on $\Delta$ such that 
\[
 (g|_{\Delta})^* \dualizing{\Delta_0}(-A) = \dualizing{\Delta} \quad \text{and} \quad \supp{A} = \exc{g|_{\Delta}}
\]
\end{itemize}
\end{lem}
\begin{proof}
Since $\Delta_0$ is the only divisor contracted by $\Phi$, we have
\begin{align}\label{rami}
 g^* \can{X} + F = \can{\tilde X} = f^*\can{\Pn 3} + a \Delta + H
\end{align}
for some $a \geq 1$, and some effective divisor $H$ that is exceptional with respect to both $f$ and $g$. Plugging $f^*\sO{\Pn 3}{1} = g^*L - E$ into \eqref{rami}, yields 
\begin{align*}
 a\Delta &\sim 4 g^*L - 4E - 2 g^*L + F - H = 2 g^*L - 4E + F - H.
\end{align*}
Since $\Delta_0 \in |L|$ by \cref{DeltaToH}, this shows that $a=2$, hence that
\[
 G = g^*\Delta_0 - \Delta = 2 E - \tfrac F 2 + \tfrac H 2.
\]
Therefore, using $E \geq F$ and $H \geq 0$,
\begin{align*}
 G - F \geq \frac F 2 + \frac H 2 \geq \frac F 2.
\end{align*}
This shows $G-F \geq 0$ since $F \geq 0$. It also shows that
\[
 \supp{G-F} \supset \supp F = \exc{g}.
\]
Since, of course, $\supp{G-F} \subset \exc g$ also holds, we have, in fact, equality.\par 
In order to show the second claim, set $A = (G-F)|_{\Delta}$. This divisor is certainly an effective Cartier divisor on $\Delta$, because $\Delta \not\subset \exc g = \supp{G-F}$, and it satisfies 
\begin{align*}
  \dualizing{\Delta} &= \dualizing{\tilde X}(\Delta)|_{\Delta} = (g^*\dualizing{X}(F + g^* \Delta_0 - G))|_{\Delta} \\
  &= (g|_{\Delta})^*(\dualizing{X}(\Delta_0)|_{\Delta_0})((F-G)|_{\Delta}) =  (g|_{\Delta})^* \dualizing{\Delta_0}(-A).
\end{align*}
This implies
\[
 \supp A \subset \exc{g|_{\Delta}}
\]
and, using $\supp{G-F} = \exc g$, we also get  
\[
 \exc{g|_{\Delta}} \subset \exc{g} \cap \Delta = \supp{G-F} \cap \Delta = \supp{A},
\]
so that, indeed, $\supp A = \exc{g|_{\Delta}}$.
\end{proof}
Now compose $\tilde g$ with an embedded resolution of $\Delta \subset \tilde X$ as in \cref{res2}. We use the notation from the following commutative diagram:
\[ 
  \begin{tikzcd}[row sep = 1em]
     \phantom{.}& \hat \Delta \arrow[hookrightarrow]{r} \arrow[swap]{dd}{\tau} \arrow[swap]{dl}{\mu} \arrow[bend left, near start]{dddd}{\phi} & \hat X \arrow{dd} \arrow[bend left]{dddd}{g}\\
     \tilde \Delta \arrow[swap]{dr}{\nu} & &\\
     \phantom{.}& \Delta \arrow[hookrightarrow, crossing over]{r} \arrow{dd} & \tilde X \arrow[swap]{dd}{\tilde g}\\
     \phantom{.}& & \\
     \phantom{.}& \Delta_0 \arrow[hookrightarrow]{r} & X
  \end{tikzcd} 
\]
Here $\fun{\nu}{\tilde \Delta}{\Delta}$ denotes the normalization. I claim that this $g$ satisfies all the properties asserted in \cref{niceResolution}.\par 
That it satisfies properties (\ref{prop1}) and (\ref{prop2}) is clear. By \cref{dimOne}, $\ideal \Delta \not\subset \ideal{B_1}^3$, which immediately implies property (\ref{prop4}). \par
Let us now show that it satisfies property (\ref{prop5}). To this end, let $A_1 \geq 0$ on $\Delta$ be the divisor provided by \cref{res1}, i.e.\
\begin{align}\label{A1}
 \dualizing{\Delta} = (\tilde g|_{\Delta})^*\dualizing{\Delta_0}(-A_1), \quad \text{and} \quad \exc{\tilde g|_{\Delta}} = \supp {A_1},
\end{align}
and let $A_2 \geq 0$ on $\hat \Delta$ be the divisor provided by \cref{res2}, i.e.\
\begin{align}\label{A2}
 &\dualizing{\hat \Delta} = \tau^*\dualizing{\Delta}(-A_2), \quad \text{and} \\ 
 &\notag \forall C \subset \hat \Delta, \text{irreducible} \colon (\dim \tau(C) = 1 \Rightarrow C \subset \supp{A_2}).
\end{align}
The divisor  
\[
 A = \tau^*A_1 + A_2.
\]
 is effective and is easily seen to satisfy
\[
 \dualizing{\hat \Delta} = \phi^*\dualizing{\Delta_0}(-A).
\] 
In order to show that it also satisfies the second condition of (\ref{prop5}), let $C \subset \exc{\phi}$ be an irreducible curve. We have to show that $C \subset \supp A \cup \exc{\mu}$. \par
First assume $C \not\subset \exc{\tau}$. Then $\tau(C) \subset \exc{\tilde g|_{\Delta}}$ and by \eqref{A1}, 
\[
 \tau(C) \subset \supp{A_1},
\]
hence $C \subset \supp{\tau^*A_1} \subset \supp A$. \par
Now assume $C \subset \exc{\tau}$. If $\dim \tau(C) = 0$, then $C \subset \exc{\mu}$, since the normalization $\nu$ is finite. \par If $\dim \tau(C) = 1$, \eqref{A2} implies that 
\[
 C \subset \supp{A_2} \subset \supp A.
\]
Lastly we show that property (\ref{prop3}) of \cref{niceResolution} holds. Let $r$ and $M_1, \ldots, M_r$ be as in \cref{clm2} of the proof of \cref{dimOne}. I claim that $B_1^{\prime} = M_r \subset \hat \Delta$ satisfies the assertions of property (\ref{prop3}).\par
The resolution $g$ we have chosen has the property that there are precisely two exceptional divisors $E_1,E_2 \subset \hat X$ that are mapped onto the curve $B_1$. They are the strict transforms of the exceptional divisors of the first two blow-ups. From \eqref{CINTE2} and \eqref{DEFONC} it follows that $E_2$ is the divisor that $f$ maps onto $H_0$. Let $G = g^*\Delta_0 - \hat \Delta$. Then taking the multiplicity along $E_2$ of the two sides of the ramification formula
\[
 g^*\can X + F = f^* \can{\Pn 3} + 2 \Delta + F^{\prime}, 
\] 
yields 
\[
 2 = 8 - 2 \cdot \mult{E_2}{G},
\]
hence $\mult{E_2}{G} = 3$. Since $\ideal{\Delta_0} \subset \ideal{B_1}^2$ and $\ideal{\Delta_0} \not\subset \ideal{B_1}^3$, 
\[
 \mult{E_1}{G} = 2 < 3 = \mult{E_2}{G}.
\]
This shows that the strict transform $\Delta_1 \subset X_1$ of $\Delta_0 \subset X$ contains the center $B_2 \subset X_1$ of the second blow-up and that $B_2 \not\subset \sing{\Delta_1}$. Therefore each curve $C \subset \Delta$ with $g(C) = B_1$ is actually mapped isomorphically to $B_1$ by $g$. In particular this is true for the curve $B_1^{\prime} = M_r \subset \hat \Delta$. \par
Since by definition $f(M_r) = C_1$ (cf.\ \cref{clm2}), it remains to show that $f|_{M_r}$ has generic degree $1$. \par 
Since $M_r \subset \exc{g}$ and $g(M_r) = B_1$, we actually have 
\[
 M_r \subset E_1 \cup E_2.
\]
If $M_r \subset E_2$, we are done, since $f|_{E_2}$ is birational. Suppose $M_r \not\subset E_2$. Then $r=2$, there is a unique irreducible curve $C \subset \hat \Delta \cap E_1$ satisfying 
\[
 g(C) = B_1 \qquad \text{and}  \qquad C \not\subset E_2;
\]
and $C = M_2$. There is a unique curve in $E_1 \cap E_2$ that is mapped onto $B_1$ by $g$. Hence that curves is $M_1 = \tilde C_1$. Summing up, we see that $g$ maps both $M_1$ and $M_2$ isomorphically to $B_1$. The morphism $g|_{E_1}$ contracts the general fibre of $f|_{E_1}$ by \cref{gContracts} and $f|_{M_1}$ and $g|_{M_2}$ are generically one-to-one. Therefore, $f$ is also generically one-to-one on $B_1^{\prime} = M_2$ by \cref{genOneToOne}. This finishes the proof of \cref{niceResolution}
\subsection{Finishing the proof}
The next lemma, together with what we already know, finally shows that $X \setminus \Delta_0 \simeq \C^3$.
\begin{lem}\label{isCompactifi}
 Given \cref{projPlus}, the meromorphic map
\[
 \Phi^{-1} \colon \Pn 3 \dashrightarrow X.
\]
is holomorphic on $\Pn 3 \setminus H_0$. 
\end{lem}
\begin{proof}
Let, as usual, $\fun{g}{\tilde X}{X}$ be a resolution of $\Phi$ that satisfies 
\[
 g(\exc g) \subset \bs L
\]
and set $f = \Phi \circ g$. Then the conclusion of the lemma is implied by $f(\exc f) \subset H_0$. Suppose this is false. Take an irreducible component
\[
 A \subset f(\exc f), \quad \text{with} \quad A \not \subset H_0
\]
and let 
\[
 B = \bigcup_{\substack{E \subset \tilde X \text{ divisor} \\ f(E) = A}} E.
\]
Since $f(\Delta) \subset H_0$ and all $f$-exceptional divisors different from $\Delta$ are $g$-exceptional by \cref{onlyOne}, we have $B \subset \exc{g}$.\par
Observe that $\dim A = 1$ because otherwise a general $H \in |\sO{\Pn 3}{1}|$ would not intersect $A$, which contradicts $g(B) \subset g(\exc g) \subset \bs L$. \par
\begin{claim}
The analytic subset $g(B)$ is connected and it is not an isolated point of $\bs L$.
\end{claim}
\begin{proof}[Proof of the claim]
It is easy to see that $B$ is connected (cf.\  proof that \eqref{connected5} in \cref{onlyOne} is connected), hence so is $g(B)$. \par 
Suppose $g(B) = \{ x \}$ is an isolated point of $\bs L$. Then in particular 
\[
 E_0 = f_*^{-1}H_0 \not\subset B
\]
since $\dim g(E_0) = 1$ by \cref{dimOne}. Therefore all divisors $E \subset g^{-1}(x)$ are $f$-exceptional, which implies that $\dim f(g^{-1}(x)) \leq 1$. Thus for general $H_1,H_2 \in |\sO{\Pn 3}{1}|$, we have
\begin{align}\label{HHFE}
 H_1 \cap H_2 \cap f(g^{-1}(x)) = \emptyset.
\end{align}
Letting $D_1,D_2 \in |L|$ denote the divisors corresponding to $H_1,H_2$, \eqref{HHFE} implies that 
\[
 x \not\in C_{D_1,D_2} = g(f_*^{-1}(H_1 \cap H_2)).
\]
On the other hand, we know from \cref{movingCurve} that for general $D_1,D_2 \in |L|$, 
\[
 D_1 \cap D_2 = C_{D_1,D_2} \cup \bs L.
\]
Since $D_1 \cap D_2$ is purely of dimension $1$, this shows that there exists an irreducible curve $B_0 \subset \bs L$ such that $x \in B_0$, which is a contradiction.
\end{proof}
Considering lines connecting a general point of $A$ and a general point of $H_0$, it is now easy to see that there exist $D_1,D_2 \in |L|$ such that the movable curve $C_{D_1,D_2}$ intersects the $1$-dimensional part of $\bs L$ in at least two points\footnote{Note that we cannot immediately conclude that this is true for general $D_1,D_2 \in |L|$, which would be a contradiction to our assumption.}, one of which lies on $B_1$. Since $\com{0}{C_{D_1,D_2}}{Q_{D_1,D_2}} = 2$ by \cref{leq2}, this shows that $D_1 \cdot D_2$ is generically reduced along $B_1$, hence this is true for general $D_1,D_2 \in |L|$ and this contradicts \cref{mib}.
\end{proof}
We can now reduce the task of studying the topology of the smooth threefold $X$ to that of studying the topology of the singular surface $\Delta_0$.
\begin{lem}\label{b3Xb3Delta}
Under \cref{projPlus} we have
\[
 \betti 2 {\Delta_0} = \betti 2 {X} = 1 \quad \text{and} \quad \betti{3}{\Delta_0} = \betti{3}{X}.
\]
\end{lem}
\begin{proof}
The first part of \cref{DeltaToH} implies that  $\bs{L} \subset \Delta_0$, hence $\Phi^{-1}_* H_0 \subset \Delta_0$. Together with the second part and the previous lemma we get that the holomorphic map
\[
 \fun{\Phi_{X \setminus \Delta_0}}{X \setminus \Delta_0}{\Pn{3} \setminus H_0}
\]
is an isomorphism. For all $k \in \N$ there is, e.g.\ by \cite[Chapter II, 10.3]{Bredon}, an exact sequence
\[
 \begin{tikzcd}
  \cCoh{k}{X \setminus \Delta_0}{\C} \arrow{r} & \Com{k}{X}{\C} \arrow{r} & \Com{k}{\Delta_0}{\C} \arrow{r} & \cCoh{k+1}{X \setminus \Delta_0}{\C}.
 \end{tikzcd}
\]
 Since $X \setminus \Delta_0 \simeq \Pn{3} \setminus H_0$ is the affine $3$-space,
\[
 \cCoh{k}{X \setminus \Delta_0}{\C} = 0, \quad \text{for} \quad k \in \{1, \ldots, 5\},
\]
hence the lemma follows.
\end{proof}
Thus we are reduced to the problem of bounding $\betti{3}{\Delta_0}$ and we finish the proof of \cref{projective} by using the resolution from \cref{niceResolution} to show that
\begin{align}\label{ineq1}
 \betti{3}{\Delta_0} \leq 1.
\end{align}
Note that $f$ restricts to a morphism $\hat \Delta \to C_1$. Since $f|_{B_1^{\prime }}$ is the normalization of $C_1$ by property (\ref{prop3}) of \cref{niceResolution} and since $\hat \Delta$ is normal, there is a factorization
\begin{align}\label{factorizationB1}
 \hat \Delta \to B_1^{\prime} \to C_1
 \end{align}
of $f|_{\hat \Delta}$ such that the restriction of the first map to $B_1^{\prime}$ is just the identity. In particular the first morphism in \eqref{factorizationB1} has connected fibres -- otherwise the Stein factorization would tell us that the restriction 
\[
 B_1^{\prime} \subset \hat \Delta \to B_1^{\prime}
\]
is not the identity. Since $\hat \Delta$ is smooth, the general fibre is smooth and since the general fibre is contained in a $1$-dimensional fibre of $f$, it is actually a $\Pn{1}$. Therefore, by contracting $(-1)$-curves contained in fibres of the first morphism in \eqref{factorizationB1} as long as they exist, we see that this morphism factors through a Hirzebruch surface 
\[
\begin{tikzcd}
    \hat \Delta \arrow{r}{h} & \hir e \arrow{r}{\pi} & B_1^{\prime} \simeq \Pn 1.
\end{tikzcd}
\]
Because $h$ is a birational morphism of smooth surfaces we have 
\begin{align}\label{b3DeltaIs0}
 \betti{3}{\hat \Delta} = \betti{3}{\hir e} = 0.
\end{align}
For an arbitrary birational map $\hat \Delta \to S$ to a non-normal surfaces $S$ the third Betti number might, of course, increase. We show that this is not the case with $\Delta_0$. To this end, set $N = \phi(\exc{\phi})$ and note that $\phi^{-1}(N) = \exc{\phi}$. Therefore, by \cref{MayViet}, there is an exact sequence
\[
 \begin{tikzcd}[column sep = 1.5em]
 \Com{2}{\Delta_0}{\C} \arrow{r} &\Com{2}{\hat \Delta}{\C} \oplus \Com{2}{N}{\C} \arrow{r} & \Com{2}{\exc{\phi}}{\C} \arrow{r} &
 \Com{3}{\Delta_0}{\C} \arrow{r} & 0,
 \end{tikzcd}
\]
where we use that $\betti{3}{\hat \Delta} = 0$ by \eqref{b3DeltaIs0} and that $\betti{3}{N} = 0$ for dimensional reason. Using the fact that the Euler characteristic of an exact sequence\footnote{That is, the alternating sum of the dimensions of its entries.} vanishes, we get 
\begin{align}\label{ineq2}
 \betti 3 {\Delta_0} &\leq \betti 2 {\exc{\phi}} - \betti 2 N - \betti 2 {\hat \Delta} + \betti 2 {\Delta_0}  \\
 &=  \betti 2 {\exc{\phi}} - \betti 2 N - \betti 2 {\hat \Delta} + 1. \notag
\end{align}
 Write the set of irreducible $1$-dimensional components of $\exc{\phi}$ as
\[
 \operatorname{1-dim}(\exc{\phi}) = \{ \hat C_1, \ldots \hat C_q, \hat f_1, \ldots, \hat f_r, R_1, \ldots, R_s\},
\]
such that
\begin{align*}
&\dim \pi ( h(\hat C_i) )= 1, \quad \text{for all} \quad i \in \{1, \ldots, q \},\\
&\dim \pi ( h(\hat f_j) )= 0, \quad \text{and} \quad \hat f_j \not\subset \exc h, \quad \text{for all} \quad j \in \{1, \ldots, r\} \quad \text{and}\\
& R_k \subset \exc{h} \quad \text{for all} \quad k \in \{1, \ldots,s\}
\end{align*} 
and let $p_1, \ldots p_r \in B_1^{\prime}$ be such that 
\[
 \pi (h(\hat f_j)) = \{ p_j \}, \quad \text{for all} \quad  j \in \{1, \ldots, r\}.
\]
Furthermore write
\begin{align*}
 \operatorname{1-dim}(\exc{h}) &= \{ R_1, \ldots, R_s, R_{s+1}, \ldots, R_{s+t}\}, \\
 \operatorname{1-dim} N &= \{ N_1, \ldots, N_u \}.
\end{align*}
Then
\[
  \betti{2}{\exc{\phi}} = q+r+s, \qquad \betti{2}{\hat \Delta} = 2 + s+t, \qquad \betti 2 N = u.
\] 
So in order proof $\betti 3 {\Delta_0} \leq 1$, it is, by (\ref{ineq2}), enough to show that
\[
 q + r + s - u - ( 2 + s + t ) = q - 2 + r - (u + t) \leq 0.
\]
We achieve this by showing 
\[
 q \leq 2 \qquad \text{and} \qquad  r \leq u + t
\]
and start by showing the first inequality. This is a consequence of the fact that $\phi^* \dualizing{\Delta_0}$ has non-trivial global sections and property (\ref{prop5}) of \cref{niceResolution}.\par
Indeed, we have $\supp A \subset \exc{\phi}$. So we can write 
\[
 A = \alpha_1 \hat C_1 + \cdots + \alpha_{q^{\prime}} \hat C_{q^{\prime}} + A^{\prime} 
\]
with $q^{\prime} \leq q$, $\alpha_i \in \Z_{\geq 1}$ and an effective divisor $A^{\prime}$ that does not contain any of the $\hat C_1, \ldots, \hat C_q$. Looking at the commutative square
\[
\begin{tikzcd}
    \hat \Delta \arrow{r}{h} \arrow{d}{\mu} & \hir e \arrow{d}{\pi}\\
    \tilde \Delta \arrow{r} & B_1^{\prime}
\end{tikzcd}
\]
and recalling that, for all $i \in \{1, \ldots, q\}$, $\hat C_i$ is not contracted by $\pi \circ h$ we see that $\hat C_i \not\subset \exc{\mu}$. Since by property (\ref{prop5}) of \cref{niceResolution}
\[
 \hat C_i \subset \exc{\mu} \cup \supp A, 
\]
we get $\hat C_i \subset \supp A$. This shows that $q = q^{\prime}$. \par
Let $R$ denote the effective divisor satisfying 
\[
 \dualizing{\hat \Delta} = h^*\dualizing{\hir e}(R),
\]
so that we have
\[
 \phi^* \dualizing{\Delta_0} = \dualizing{\hat \Delta}(A) = h^*\dualizing{\hir e}(A+R).
\]
This entails the following inclusion
\[
\Com{0}{\hat \Delta}{\phi^* \dualizing{\Delta_0}^{-1}} \subset \Com{0}{\hat \Delta}{h^* \dualizing{\hir e}^{-1}(- \sum_{i=1}^q \hat C_i)} = \Com{0}{\hir e}{\dualizing{\hir e}^{-1}(- \sum_{i=1}^q C_i)},
\]
where we set $C_i = h(\hat C_i)$. A line bundle $L^{\prime}$ on the Hirzebruch surface $\hir e$ has no non-trivial sections if $L^{\prime} \cdot \f < 0$. Therefore, since 
\[
 \dualizing{\hir e}^{-1}(- \sum_{i=1}^q C_i) \cdot \f = 2 - \sum_{i=1}^q C_i \cdot \f \leq 2 - q
\]
and since 
\[
 \com{0}{\hat \Delta}{\phi^*\dualizing{\Delta_0}^{-1}} \geq \com{0}{\Delta_0}{\dualizing{\Delta_0}^{-1}} = \com{0}{\Delta_0}{L|_{\Delta_0}} = \com 0 X L - 1 = 3,
\]
we must have $q \leq 2$. \par 
We now show $r \leq u+t$. Certainly $B_1 \subset N$, as $B_1 \subset \sing{\Delta_0}$. Assume that $N_1 = B_1$, i.e.\ 
\[
 \operatorname{1-dim}(N) = \{ B_1, N_2, \ldots, N_u \}.
\]
Since $\pi \circ h(B_1^{\prime}) = B_1^{\prime}$, we have $B_1^{\prime} \not\subset h^{-1}(f_i)$, for all $i$, where $f_i = h(\hat f_i)$. If there exists 
\[
 B_1^{\prime} \neq B_1^{\prime\prime} \subset \hat \Delta
\]
with $g(B_1^{\prime\prime}) = B_1$, we can arrange that $B_1^{\prime\prime} \not\subset h^{-1}(f_i)$ for all $i \geq 2$. Then, by property (\ref{prop4}) of \cref{niceResolution},
\begin{align}\label{notB1}
 B_1 \not\subset \phi ( h^{-1}(f_i)) \quad \text{for all} \quad i \geq 2.
\end{align}
The inequality $r \leq u+t$ is shown if we construct a map 
\[
 \fun{\alpha}{\{f_2 , \ldots, f_r \}}{\{ N_2, \ldots, N_u, R_{s+1}, \ldots, R_{s+t}\}}
\]
and show that it is injective. First recall that
\[
 \fun{\phi \circ (\pi \circ h|_{B_1^{\prime}})^{-1}}{B_1^{\prime}}{B_1}
\]
is an ismorphism and set, for all $i \in \{1, \ldots r\}$, 
\[
q_i = \phi \circ (\pi \circ h|_{B_1^{\prime}})^{-1}(p_i).
\]
Now let $i \in \{2, \ldots, r\}$. The preimage $h^{-1}(f_i)$ is not contracted by $\phi$ because otherwise we would have the following contradiction
\[
 0 > (h^* (f_i))^2 = (f_i)^2 = 0.
\]
Since the morphism $h$ has connected fibres, $h^{-1}(f_i)$ is connected and there exists an irreducible curve $\hat Q_i \subset h^{-1}(f_i)$ that is not contracted by $\phi$ and with 
\[
 q_i \in \phi(\hat Q_i).
\]
We set $Q_i = \phi (\hat Q_i)$. If $Q_i \subset N$, we have $Q_i \neq B_1$ by \eqref{notB1} and we define 
\[
 \alpha(f_i) = Q_i.
\]
If $Q_i \not\subset N$ then $\hat Q_i \not\subset \exc{\phi}$ and in particular
\[
 \hat Q_i \not\subset \hat f_i \cup R_1 \cup \ldots \cup R_s,
\]
which shows that $\hat Q_i \in \{R_{s+1}, \ldots, R_{s+t}\}$. Then we define
\[
 \alpha(f_i) = \hat Q_i.
\]
It remains to show that $\alpha$ is injective. Let $i,j \in \{2, \ldots,r\}$ and assume 
\[
 \alpha(f_i) = \alpha(f_j).
\]
First consider the case that $\alpha(f_i) \in \{R_{s+1}, \ldots R_{s+t}\}$. Then, in the notation from above,
\[
 \hat Q_i = \alpha(f_i)  = \alpha(f_j) = \hat Q_j,
\]
hence
\[
 \{p_i\} = \pi \circ h(\hat Q_i) = \pi \circ h(\hat Q_j) = \{p_j\},
\]
which shows that $f_i = f_j$. \par
Now consider the case that $\alpha(f_i) \in \{ N_2, \ldots, N_u \}$ and assume $f_i \neq f_j$ in order to derive a contradiction. Then $q_i \neq q_j$ and since $Q_i = \alpha(f_i) = \alpha(f_j)= Q_j$, we have
\[
 \{q_i,q_j\} \subset Q_i \cap B_1,
\]
which shows that 
\begin{align}\label{tooBig}
\com{1}{Q_i+B_1}{\sg{Q_i+B_1}} \geq 1.
\end{align}
Let $p \in C_1 \subset \Pn 3$ be the image of $p_i$ under the normalization map $B_1^{\prime} \to C_1$, in other words $f(\hat Q_i) = \{p\}$. Let $H_1,H_2 \in |\sO{\Pn 3}{1}|$ be such that 
\[
 p \in H_1 \cap H_2 \qquad \text{and} \qquad H_1 \cap H_2 \not\subset H_0,
\]
and let $D_1,D_2 \in |L|$ be the divisors that correspond to $H_1,H_2$ via $\Phi$. Then we have
\[
 Q_i + B_1 \subset D_1 \cdot D_2 
\]
and by \cref{h1eq1} and \eqref{tooBig} equality holds. This is a contradiction, since also  $C_{D_1,D_2} = \Phi_*^{-1}(H_1 \cap H_2) \subset D_1 \cdot D_2$. Therefore $\alpha$ is injective and $r \leq u+t$, hence $\betti{3}{\Delta_0} \leq 1$. \par
Together with \cref{b3Xb3Delta}, this shows that $\betti{3}{X} \leq 1$. Since Hodge decomposition holds on $X$, we have, in fact, that $\betti 3 X = 0$. This finishes the proof of \cref{projective}. 
\raggedright
\printbibliography
\end{document}